\documentclass[11pt]{article}
\usepackage[toc,page]{appendix}
\usepackage{bbm}
\usepackage{fullpage}
\usepackage[top=1in, bottom=1.25in, left=1in, right=1in]{geometry}
\usepackage{etaremune}
\usepackage{enumitem}
\usepackage{graphicx}
\usepackage[labelformat=empty]{caption}
\usepackage{subcaption}
\usepackage{tikz}
\usepackage{amssymb,amsmath,amsthm}
\usepackage{hyperref}
\usepackage[capitalise,nameinlink]{cleveref}
\crefname{thm}{Theorem}{Theorems}
\crefname{lem}{Lemma}{Lemmas}
\crefname{rmk}{Remark}{Remarks}
\usepackage[linesnumbered,lined,ruled]{algorithm2e}
\usepackage{comment}
\usepackage{mathabx}

\usetikzlibrary{decorations.pathreplacing,calligraphy}

\DeclareSymbolFont{rsfs}{U}{rsfs}{m}{n}
\DeclareSymbolFontAlphabet{\mathscrsfs}{rsfs}
\usepackage{mathrsfs}
\usepackage{url}
\usepackage{times}

\hypersetup{
  colorlinks   = true,
  urlcolor     = blue,
  linkcolor    = blue,
  citecolor   = red
}

\newtheorem*{rep@theorem}{\rep@title}
\newcommand{\newreptheorem}[2]{%
\newenvironment{rep#1}[1]{%
 \def\rep@title{#2 \ref{##1}}%
 \begin{rep@theorem}}%
 {\end{rep@theorem}}}
\makeatother

\newtheorem{thm}{Theorem}[section]
\newtheorem{lem}[thm]{Lemma}

\newreptheorem{ppn}{Proposition}

\theoremstyle{definition}

\newtheorem{rmk}[thm]{Remark}

\numberwithin{equation}{section}

\def\beq#1\eeq{%
    \begin{equation}%
    #1%
    \end{equation}%
}
\def\baln#1\ealn{%
    \begin{align*}%
    #1%
    \end{align*}%
}
\def\balnn#1\ealnn{%
    \begin{align}%
    #1%
    \end{align}%
}

\def\fr{\frac}
\def\lt{\left}
\def\rt{\right}
\def\la{\langle}
\def\ra{\rangle}
\def\<{\langle}
\def\>{\rangle}

\def\eps{{\varepsilon}}

\def\norm#1{\lt\|#1\rt\|}

\def\bbE{{\mathbb{E}}}

\def\bbP{{\mathbb{P}}}

\def\bbR{{\mathbb{R}}}

\DeclareMathOperator*{\EE}{\bbE}
\DeclareMathOperator*{\PP}{\bbP}

\DeclareMathOperator*{\argmin}{arg\,min}

\def\cE{{\mathcal{E}}}

\def\cN{{\mathcal{N}}}

\def\cR{{\mathcal{R}}}

\def\be{{\boldsymbol e}}
\def\bg{{\boldsymbol g}}
\def\bm{{\boldsymbol m}}

\def\bu{{\boldsymbol u}}
\def\bv{{\boldsymbol v}}
\def\bw{{\boldsymbol w}}
\def\bx{{\boldsymbol x}}
\def\by{{\boldsymbol y}}

\def\bV{{\boldsymbol V}}
\def\bA{{\boldsymbol A}}
\def\bE{{\boldsymbol E}}
\def\bG{{\boldsymbol G}}

\def\bW{{\boldsymbol W}}
\def\bZ{{\boldsymbol Z}}
\def\tbW{\tilde{\boldsymbol W}}

\def\tW{\tilde{W}}

\def\bDel{{\boldsymbol \Delta}}

\def\hm{\hat{m}}

\def\Err{{\sf Err}}

\def\bfone{{\boldsymbol 1}}

\def\by{{\boldsymbol y}}

\def\bsig{{\boldsymbol \sigma}}
\def\brho{{\boldsymbol \rho}}

\def\btau{{\boldsymbol \tau}}

\def\bB{{\boldsymbol B}}
\def\bI{{\boldsymbol I}}

\def\bLambda{{\boldsymbol \Lambda}}
\def\bq{{\boldsymbol q}}

\def\txi{{\widetilde \xi}}

\def\tg{{\widetilde g}}

\def\hbm{{\widehat \bm}}

\def\bR{{\boldsymbol R}}

\def\de{{\mathsf{d}}}

\def\Band{{\mathsf{Band}}}

\def\Tr{{\mathsf{Tr}}}

\def\sTAP{{\mbox{\tiny\sf TAP}}}

\def\op{\mbox{{\tiny \rm op}}}

\definecolor{darkblue}{RGB}{25,25,200}

\definecolor{darkpurple}{RGB}{100,0,100}

\author{
    Brice Huang\thanks{Department of Statistics, Stanford University}
    \and
    Andrea Montanari\thanks{Department of Statistics and Department of Mathematics, Stanford University}
    \and
    Huy Tuan Pham\thanks{Department of Mathematics, University of Chicago}
}

\title{Approximating the Gibbs Measure of Spherical Spin Glasses to Within
$o(N^{-1/2})$ in the Replica Symmetric Phase}

\begin{document}

\maketitle

\begin{abstract}
We study the magnetization vector of mixed spherical spin glasses in a
weak external field, in the replica symmetric phase.  We derive a
coordinatewise expansion of the magnetization which includes the
leading cubic correction to the quadratic approximation, with an error
that is $o(N^{-1/2})$ per coordinate.  We also prove a corresponding
high-moment bound for the full vector.  Finally, our companion paper
applies these estimates
to the effective Hamiltonian obtained by restricting a tilted spherical
spin glass to a narrow band around a TAP critical point.  The resulting
sharp approximation is designed for applications in which normalized
Wasserstein accuracy is insufficient, including total-variation
sampling by stochastic localization.
\end{abstract}

\tableofcontents

\section{Introduction}

Let $S_N=\{\bsig\in\bbR^N:\|\bsig\|_2^2=N\}$ and let $\mu_0$ be its
uniform probability measure.  We consider the mixed spherical
Hamiltonian
\begin{align}
\label{eq:def-HN}
H_N(\bsig)
&=\sum_{p\geq2}\frac{\gamma_p}{N^{(p-1)/2}}
\sum_{i_1,\ldots,i_p=1}^N
g_{i_1,\ldots,i_p}\sigma_{i_1}\cdots\sigma_{i_p},
&
g_{i_1,\ldots,i_p}&\stackrel{\rm iid}{\sim}\cN(0,1),
\end{align}
with mixture $\xi(s)=\sum_{p\geq2}\gamma_p^2s^p$.  Thus
\[
\EE H_N(\bsig^1)H_N(\bsig^2)
=N\xi(\langle\bsig^1,\bsig^2\rangle_N),
\]
where $\langle \bsig^1, \bsig^2 \rangle_N = \langle \bsig^1, \bsig^2 \rangle / N$ is the normalized overlap,
and $\|\bx\|_N=\|\bx\|/\sqrt{N}$ is the corresponding normalized norm.
The principal object of this paper is the magnetization of the weakly
tilted Gibbs measure
\[
\mu_{H_N,\bu}(\de\bsig)
=\frac{1}{Z_{N,\bu}}\exp\{H_N(\bsig)+\langle\bu,\bsig\rangle\}
\mu_0(\de\bsig),
\]
where $\bu$ is a small tilt in the sense that $\|\bu\| \le N^{c_0}$ for a small constant $c_0$.
Our goal is an approximation sharper than the natural
$N^{-1/2}$ coordinate scale.

The basic replica-symmetry assumption used for the local estimate is
\begin{equation}
\label{eq:replica-symmetry}
\xi''(0)<1,\qquad
\xi(q)+q+\log(1-q)<0\quad\text{for every }q\in(0,1).
\end{equation}
As explained in the companion paper \cite{huang2026sampling}, this is a strict version of the following criterion, under which the model's free energy is replica symmetric \cite{talagrand2006spherical}:
\[
\xi(q) + q + \log(1-q) \le 0 \quad \text{for every }q\in [0,1).
\]
The companion \cite{huang2026sampling} assumes the condition
\begin{equation}
\label{eq:amp-works}
\xi''(s)<\frac{1}{(1-s)^2}\quad\text{for every }s\in[0,1),
\end{equation}
which implies \eqref{eq:replica-symmetry} by integrating twice.

The first main result, Theorem~\ref{lemma:LocalEffective1}, gives an
explicit coordinatewise expansion: the resolvent response to the
external field plus the leading contribution from the cubic disorder.
Theorem~\ref{lem:correction-band} strengthens this to a high-moment
estimate for the whole magnetization vector.  This paper
proves these statements for a general weak-field model.

In the companion paper \cite{huang2026sampling}, these results are used to estimate the magnetization of a Gibbs measure, arising from stochastic localization, with a macroscopic external field.
In that paper, it is first shown that this Gibbs measure is concentrated on a spherical band of vectors approximately orthogonal to a TAP fixed point $\bm^{\sTAP}$.
This band can be viewed as a union of codimension-two spherical slices orthogonal to $\bm^{\sTAP}$, which are each approximated by a mixed $p$-spin model with zero or small magnetic field.
This paper's results are used to estimate the magnetizations of these models.

\subsection{Related work}

The motivation for the precision pursued here comes in part from
algorithmic stochastic localization. This approach was introduced for
spin-glass sampling in \cite{alaoui2022sampling}, building on Eldan's
stochastic-localization method
\cite{eldan2013thin,eldan2020taming,eldan2022analysis}. Closely related
processes appear in machine learning as denoising diffusions
\cite{sohl2015deep,ho2020denoising,song2021score}; see
\cite{montanari2023sampling} for further discussion of this connection.
In stochastic localization, the drift is the magnetization of a
Gaussian tilt of the target measure. Consequently, controlling the law
of a long stochastic localization trajectory requires an estimate of this
magnetization that is much sharper than the usual extensive-scale
control.

Earlier analyses of algorithmic stochastic localization for mean-field
spin glasses established sampling guarantees in normalized Wasserstein
distance \cite{alaoui2022sampling,celentano2024sudakov,alaoui2023sampling}.
Their mean estimator is based on the Thouless--Anderson--Palmer (TAP)
free energy \cite{thouless1977solution}, with approximate message
passing selecting the relevant TAP critical point. These works control
the squared error of the TAP estimate at the $o(N)$ scale. This is the
standard precision for macroscopic observables in spin-glass theory
\cite{mezard1987spin,talagrand2010mean}, but it does not suffice for
controlling the total variation error of the algorithm. The present paper isolates the local calculation behind
that stronger conclusion: the TAP center must be supplemented by an
explicit correction arising from the cubic part of the effective
Hamiltonian.

Our calculation also belongs to the broader study of spherical spin
glasses in the replica symmetric phase. For spherical quadratic models,
contour and Laplace-transform methods give precise information on
partition functions and their fluctuations
\cite{baik2016fluctuations}. We extend this type of analysis to
coordinate observables and then perturb away from the quadratic model,
keeping the leading contribution of the cubic disorder. The
high-temperature, or strictly replica symmetric, condition used here is
closely related to the overlap estimates and constructive analysis in
\cite{huang2023constructive}.

\section{Main Results}
\label{sec:FirstLocal}

We consider the Hamiltonian of Eq.~\eqref{eq:def-HN}, which we rewrite as
\begin{align}
H(\bsig)& =\<\bu,\bsig\> +H_{\ge2}(\bsig) = \<\bu,\bsig\> + \sum_{p\ge 2}H_p(\bsig) \,,\label{eq:FirstLocal}\\
H_{p}(\bsig)&=\frac{1}{N^{(p-1)/2}}\sum_{i_{1},\ldots,i_{p}=1}^Ng_{i_{1},\ldots,i_{p}}\sigma_{i_{1}}\dots\sigma_{i_{p}},
   \qquad
     g_{i_1,\ldots,i_p} \stackrel{i.i.d.}{\sim} \cN(0,\beta_p^2)\, .
\end{align}
We further $H_{\ge i}(\bsig) = \sum_{p\ge i}H_p(\bsig)$. We recast the mixture of
$H_{\ge2}$ as $\xi(s) = \sum_{p\ge 2} \beta_p^2 s^p$.

The results in this paper hold for all models satisfying the replica symmetry condition \eqref{eq:replica-symmetry}.
Note that the first inequality in \eqref{eq:replica-symmetry} implies $\beta_2^2 < 1/2$.

We will always assume $\|\bu\|_{2}\le c_0\sqrt{N}$, with $c_0$ a small constant, and in some lemmas $\|\bu\|_{2}\le N^{c_0}$.
We will denote by $\mu(\de\bsig) \propto \exp(H(\bsig))\,\mu_0(\de\bsig)$ the corresponding Gibbs measure.

Note that we can view $H_{2}$ as a quadratic form with $H_2(\bsig)=\<\bsig,\bW^{(2)}\bsig\>$
(with entries $W^{(2)}_{ij}=(g_{ij}+g_{ji})/2\sqrt{N}$).
Hence $\bW^{(2)}$ is a GOE matrix scaled by $\beta_2/\sqrt{2}$. We will work in the orthonormal
basis diagonalizing $\bW^{(2)}$ and let its spectrum be $\bLambda=(\Lambda_i)_{i\le N}$,
with $\Lambda_1\ge \Lambda_2\ge\cdots\ge\Lambda_N$.
We will occasionally identify $\bLambda$ with the diagonal matrix with diagonal entries $\Lambda_i$.
We also write $\bW^{(3)}$ for the symmetric 3rd-order tensor such that $H_3(\bsig) = \la \bW^{(3)}, \bsig^{\otimes 3} \ra$, written in the basis of eigenvectors of $\bW^{(2)}$.
That is, $\bW^{(3)}$ is obtained by rotating $\tbW^{(3)}$ with entries $\tW^{(3)}_{ijk}= (g_{ijk}+{\rm permutations})/6N$.

Given a symmetric matrix $\bA\in \bbR^{N\times N}$, we define
\begin{align}
    G(\gamma) = G(\gamma;\bA,\bu) &:= \gamma - \fr{1}{2N} \log \det(\gamma \bI - \bA ) + \fr{1}{4N} \<\bu,(\gamma \bI - \bA)^{-1} \bu\>,
    \label{eq:GDef}\\
    \gamma_\ast = \gamma_\ast(\bA,\bu) & = \argmin_{\gamma > \lambda_{\max}(\bA)} G(\gamma;\bA,\bu)\, . \label{eq:GammaDef}
\end{align}
Note that $G$ is convex with $\lim_{\gamma \downarrow \lambda_{\max}(\bA)} G'(\gamma) = -\infty$, $\lim_{\gamma \uparrow +\infty} G'(\gamma) = +\infty$, so $\gamma_\ast$ is also the unique solution to $G'(\gamma) = 0$.
We will omit the argument $\bA$ or $\bu$
whenever clear from the context (in particular, we typically omit $\bu$ and omit $\bA$ when $\bA=\bLambda$
is the diagonal matrix containing the eigenvalues of $\bW$).
\begin{thm}\label{lemma:LocalEffective1}
\label{lem:spin}
There exists $c_0>0$ such that, for $\|\bu\|\le N^{c_0}$,
and under the additional assumptions above, the following
holds. Let $\gamma_*=\gamma_*(\bLambda)$ and define, for $j\le N$
\begin{align}
    \hm_{j}:= \frac{u_{j}}{2(\gamma_{*}-\Lambda_{j})}
    +\frac{3}{4(\gamma_{*}-\Lambda_{j})} \sum_{i=1}^N
    \frac{W^{(3)}_{jii}}{\gamma_{*}-\Lambda_{i}}\, .
\end{align}
Set $\hbm:=(\hm_j)_{j\le N}$.
Then, for some $c>0$, for all $i\in [N]$, with probability $1-N^{-c}$,
\[
\int\sigma_{i}\,\mu(\de\bsig)= (1+O(N^{-c}))\Big(\hm_i+O(N^{-1/2-c})\Big).
\]
\end{thm}

Together with further estimates, we will use Theorem \ref{lem:spin}
to prove the following theorem, which is the main result of the paper.
\begin{thm}
\label{lem:correction-band}
Let $\alpha \ge 2$. There exists $c_0>0$ such that, for $\|\bu\|\le N^{c_0}$, and under the additional assumptions above,
we have for some $c>0$ that
\begin{align}
\bbE\left[\left\|\int\bsig\,\mu(\de\bsig)-\hbm\right\|_{N}^{\alpha}  \right] = O(N^{-\alpha/2-c}).
\end{align}
\end{thm}

\begin{rmk}[Uniformity in the mixture]\label{rmk:uniform-mixtures}
Fix $\kappa,M>0$ and $R>1$, and consider mixtures satisfying
\begin{align}\label{eq:uniform-mixtures}
\sum_{p\ge2}\beta_p^2R^p\le M,\qquad
\xi(s)+s+\log(1-s)\le-\kappa s^2\quad (0<s<1).
\end{align}
The constants and thresholds for $N$ in \cref{lem:spin,lem:correction-band} and the lemmas below can be chosen depending only on $\kappa,M,R$ and the moment parameter $\alpha$.
Indeed, the spectral and replica-symmetry gaps and the fixed-order derivative bounds are uniform. Compactness of this class in $C^2([-1,1])$ and Gaussian interpolation also make the free energy convergence used in \cref{lem:truncate} uniform. The remaining estimates use only these bounds.
\end{rmk}

The rest of this paper is devoted to the proof of Theorem~\ref{lemma:LocalEffective1} and
Theorem \ref{lem:correction-band}.

\subsection{Overview of the proof}

The first step of our proof is to give accurate control of quadratic models. In particular, using the Laplace transform and saddle-point calculations, we show both moment bounds and arbitrarily high-precision expansions of moments of overlaps in models with only linear and quadratic terms in the Hamiltonian; see Lemmas \ref{lem:laplace}, \ref{lem:laplace-moments}, \ref{lem:laplace-moments-2}.

The next important step is a delicate high-moment calculation showing that the contribution from terms of degree at least $3$ in the Hamiltonian to the partition function can be replaced by its annealed version; see Lemma \ref{lem:high-acc-anneal}. For this moment calculation to work in the replica-symmetric regime (\ref{eq:replica-symmetry}), we need to employ some additional spin glass ideas. In particular, we first show that we can truncate to the ``typical'' set of $\bsig$ where the partition function restricted to $\bsig'$ far from orthogonal to $\bsig$ is exponentially small; see (\ref{eq:TypicalDef}) and Lemma \ref{lem:truncate}. We then show that, in computing moments of the partition function, we can truncate the integral to collections of replicas that are all typical and nearly orthogonal to each other; see Lemmas \ref{lem:key-error}, \ref{lem:OrthogonalTuples}, \ref{lemma:A-T-comparison}. After averaging over the high-degree terms in the Hamiltonian, within this region we can apply Taylor's expansion to the exponential terms involving the overlaps, and appeal to the accurate moment estimates in suitable quadratic models.

Finally, to obtain the magnetization estimate, we work in the eigenbasis of the quadratic part of the Hamiltonian, and estimate the average under the Gibbs measure of each spin, say $\sigma_1$. We use the inputs above to anneal over the contribution of terms of degree larger than $3$ in the Hamiltonian as well as cubic terms $g_{i,j,k}\sigma_i\sigma_j\sigma_k$ with $i,j,k\ne 1$; see Lemma \ref{lem:anneal}. The leading remaining contribution can be represented as an integral over $\sigma_1$ of the partition function of a suitable quadratic Hamiltonian on $\bsig_{-1}=(\sigma_2,\dots,\sigma_N)$. We again appeal to the accurate quadratic calculation, together with local expansions around the saddle point, to obtain the desired estimate of $\langle\sigma_1\rangle$ in Lemma \ref{lem:spin}. Combining this estimate with the high-probability bound in Lemma \ref{lem:high-prob} gives the high-moment estimate in Lemma \ref{lem:correction-band}.

\section{Quadratic Hamiltonians}
We begin by proving
several supporting lemmas about quadratic models. The Laplace transform allows us to compute accurately various statistics of quadratic models. We note that the use of Laplace transforms in studying spherical quadratic models has been utilized before, for example for analyzing the fluctuation of the free energy in \cite{baik2016fluctuations}. We will, however, need accurate control over a number of statistics beyond the free energy.

\begin{lem}
\label{lem:laplace}
For $\bA\in\bbR^{N\times N}$ a GOE matrix scaled by $\alpha < 1/2$, and $\bu \in \bbR^N$ such that $\|\bu\|^2 \le \eps N$ for $\eps>0$ depending only on $1/2-\alpha$,
there exists a constant $c>0$ such that, defining $G(\gamma) = G(\gamma;\bA)$ and
$\gamma_*= \gamma_*(\bA)$, we have that the following claim holds with probability at least $1-\exp(-cN)$:
\begin{align}
\int e^{\langle\bsig,\bA\bsig\rangle+\langle \bu,\bsig\rangle} \mu_0(\de\bsig)
=(1+O(N^{-c}))\sqrt{\frac{2}{G''(\gamma_*)}} \, (2e)^{-N/2}\, e^{NG(\gamma_*)},\label{eq:laplace-1}
\end{align}
and, for $\bv_k$ any eigenvector of $\bA$ (uniformly over $k$)
\begin{equation}
\frac{\int\<\bv_k,\bsig\> e^{\<\bsig,\bA\bsig\>+\< \bu,\bsig\>}\,\mu_0(\de\bsig)}{\int e^{\<\bsig,\bA\bsig\>+\< \bu,\bsig\>}\,\mu_0(\de\bsig)}=(1+O(N^{-c}))
\frac{\<\bv_k,\bu\>}{2(\gamma_{*}-\lambda_k(\bA))}\, .\label{eq:laplace-mean}
\end{equation}
\end{lem}

\begin{proof}
By a change of basis, we can assume that $\bA=\bLambda$ is diagonal (and its entries ordered).
Let
\begin{align}
E(\ell):=\ell^{N/2-1}
\int\exp\Big(\langle \bu,\bsig\rangle\sqrt{\ell}+\langle\bLambda,\bsig^{\otimes 2}\rangle\ell\Big)\,
\mu_0(\de\bsig)\, .
\end{align}
Then the Laplace transform of $E$ is given by
\[
F(t)=\int_{0}^{\infty} e^{-t\ell}\, E(\ell)\,\de\ell,
\]
and one has (for $\Re(\gamma)>\Lambda_1=\max_{i\le N}\Lambda_i$)
\[
E(\ell)=\frac{1}{2\pi i}\int_{N\gamma-i\infty}^{N\gamma+i\infty} e^{t\ell}\, F(t)\, \de t.
\]
We evaluate, for $\Re(t)>\Lambda_1$,
\begin{align*}
F(Nt) & =\int_{0}^{\infty}\int e^{-Nt\ell}\exp\Big(\langle \bu,\bsig\rangle\sqrt{\ell}+\langle\bLambda,\bsig^{\otimes 2}\rangle\ell\Big)\ell^{N/2-1}\mu_0(\de\bsig)\,  \de\ell\\
 & =\frac{\Gamma(N/2)}{(N\pi)^{N/2}}\int_{\bbR^{N}}\exp\Big(-t\|\by\|^{2}+\langle \bu,\by\rangle+\langle\bLambda,\by^{\otimes 2}\rangle\Big)\, \de\by\\
 & =\frac{\Gamma(N/2)}{N^{N/2}}\exp\left\{-\frac{1}{2}\sum_{i=1}^{N}\log(t-\Lambda_{i})+\sum_{i=1}^{N}\frac{u_{i}^{2}}{4(t-\Lambda_{i})}\right\}
 \, .
\end{align*}
Hence, by the inverse Laplace transform, for all $\gamma \in \bbR$, $\gamma > \max_{i\le N} \Lambda_i$,
\begin{align}
E(1) & =\frac{N}{2\pi i}\int_{\gamma-i\infty}^{\gamma+i\infty} e^{Nt}\, F(Nt)\,\de t\nonumber\\
 & =\frac{\Gamma(N/2)}{2\pi N^{N/2-1}}\int_{-\infty}^{\infty}\exp\big(NG(\gamma+iz)\big)\de z,\label{eq:E1Formula}
\end{align}
where $G(x) = G(x;\bLambda)$ is defined as per Eq.~\eqref{eq:GDef}.

Let $\gamma_{*}$ be defined as per Eq.~\eqref{eq:GammaDef}.
Per the discussion below \eqref{eq:GammaDef}, $\gamma_*$ is the unique solution to $G'(\gamma_*) = 0$.
Explicitly,
\[
N-\frac{1}{2}\sum_{i=1}^{N}\frac{1}{\gamma_{*}-\Lambda_{i}}-\sum_{i=1}^{N}\frac{u_{i}^{2}}{4(\gamma_{*}-\Lambda_{i})^{2}}=0.
\]
Our assumption on $\alpha$ and $\|\bu\|$ implies that $\gamma_* - \max_i \bLambda_i > \delta$ for an appropriate $\delta$ depending only on $1/2-\alpha$.
We will set $\gamma=\gamma_*$ in Eq.~\eqref{eq:E1Formula}. Note
that
\[
\Re(G(\gamma_{*}+iz)-G(\gamma_{*}))=-\frac{1}{4N}\sum_{i=1}^{N}\log(1+z^{2}/(\gamma_{*}-\Lambda_{i})^{2})-\frac{1}{4N}\sum_{i=1}^{N}\frac{u_{i}^{2}z^{2}}{(\gamma_{*}-\Lambda_{i})(z^{2}+(\gamma_{*}-\Lambda_{i})^{2})}.
\]
For $|z|\in((\log N)/\sqrt{N},1)$, we have $\Re(G(\gamma_{*}+iz)-G(\gamma_{*}))<-cz^{2}$,
and for $|z|\ge1$, we have $\Re(G(\gamma_{*}+iz)-G(\gamma_{*}))<-c\log(1+cz^{2})$.
This implies that
\[
\left|\int_{|z|>\frac{\log N}{\sqrt{N}}}\exp\Big(NG(\gamma_{*}+iz)-NG(\gamma_{*})\Big)\, \de z\right|< e^{-c(\log N)^{2}}\, .
\]
On the other hand, for $|z|\le(\log N)/\sqrt{N}$ we use the
Taylor expansion:
\begin{align*}
G(\gamma_{*}+iz)&=G(\gamma_{*})+\sum_{j=1}^{k}\frac{(iz)^{j}}{j!}G^{(j)}(\gamma_{*})+\Err_{N,k+1}\, ,\\
\Err_{N,k+1} & \le \frac{C}{k!}
\left(\frac{\log N}{\sqrt{N}}\right)^{k}\sup_{|z|\le(\log N)/\sqrt{N}}\big|G^{(k)}(\gamma_{*}+iz)\big|\, .
\end{align*}
We have that
\[
G^{(1)}(z)=1-\frac{1}{2N}\sum_{i=1}^{N}\frac{1}{z-\Lambda_{i}}-\frac{1}{4N}\sum_{i=1}^{N}\frac{u_{i}^{2}}{(z-\Lambda_{i})^{2}},
\]
\[
G^{(j)}(z)=\frac{(-1)^{j}(j-1)!}{2N}\sum_{i=1}^{N}\frac{1}{(z-\Lambda_{i})^{j}}+\frac{(-1)^{j}j!}{4N}\sum_{i=1}^{N}\frac{u_{i}^{2}}{(z-\Lambda_{i})^{j+1}}.
\]
In particular, with probability $1-\exp(-cN)$ over $\bA$, $\sup_{|z|\le(\log N)/\sqrt{N}}|G^{(j)}(\gamma_{*}+iz)|\le j!C^{j}$
for a finite constant $C>0$ as long as $\|\bu\|^{2}\le N$. Hence,
we have (for $\ell_N:=(\log N)/\sqrt{N}$ and $J_N:=[-\ell_N,\ell_N]$)
\begin{align*}
 \int_{J_N}\exp\Big(N G(\gamma_{*}+iz)-N G(\gamma_{*})\Big)\, \de z
 & =\int_{J_N} e^{- N G^{(2)}(\gamma_{*}) z^{2}/2}\,
 \exp\left(O(N\ell_N^3)\right)\, \de z\\
 & =\sqrt{\frac{2\pi}{NG^{(2)}(\gamma_{*})}}\lt(1+O(N^{-1/2+\eps})\rt).
\end{align*}
Together with Eq.~\eqref{eq:E1Formula}, we get
\begin{align*}
E(1) = \frac{\Gamma(N/2)}{2\pi N^{N/2-1}} e^{NG(\gamma_*)}
\left\{\sqrt{\frac{2\pi}{NG^{(2)}(\gamma_{*})}}+O(N^{-1+\eps})\right\}\, ,
\end{align*}
which yields \eqref{eq:laplace-1} by Stirling's formula.

By a similar argument, we obtain the integral of the spin
$\sigma_k = \langle \bv_k,\bsig\rangle$ (recall we are working in the basis in which $\bA$ is diagonal). Let
\[
E_{k}(\ell):=\ell^{N/2-1}\int \sqrt{\ell}\sigma_{k}
\exp\Big(\langle \bu,\bsig\rangle\sqrt{\ell}+\langle\bLambda,\bsig^{\otimes 2}\rangle\ell\Big)\,
\mu_0(\de\bsig).
\]
Then the Laplace transform can be evaluated as
\begin{align*}
F_{k}(Nt) & =
\int_{0}^{\infty}\int e^{-Nt\ell} \sqrt{\ell}\sigma_{k} \exp\Big(\langle \bu,\bsig\rangle\sqrt{\ell}+\langle\bLambda,\bsig^{\otimes 2}\rangle\ell\Big)\ell^{N/2-1}\mu_0(\de\bsig)\,  \de\ell \\
&=\frac{\Gamma(N/2)}{(N\pi)^{N/2}}\int_{\bbR^{N}}y_{k}\exp\left(-t\|\by\|^{2}+\langle \bu,\by\rangle+\langle\bLambda,\by^{\otimes 2}\rangle\right) \de \by\\
 & = \frac{\Gamma(N/2)}{N^{N/2}} \frac{u_{k}}{2(t-\Lambda_{k})}
 \exp\left\{-\frac{1}{2}\sum_{i=1}^{N}\log(t-\Lambda_{i})+\sum_{i=1}^{N}\frac{u_{i}^{2}}{4(t-\Lambda_{i})}\right\} .
\end{align*}
Then we apply the same strategy as for computing $E(1)$. By inverse Laplace transform:
\begin{align}
E_{k}(1)=\frac{\Gamma(N/2)}{2\pi N^{N/2-1}}u_k\, \int_{-\infty}^{\infty}
\frac{1}{2(\gamma_*+iz-\Lambda_k)}\exp\big(NG(\gamma_*+iz)\big)\de z,\label{eq:E1_kFormula}
\end{align}
We make a negligible error in restricting to $J_N:=[-\ell_N,\ell_N]$
(for $\ell_N:=(\log N)/\sqrt{N}$)
\begin{align*}
E_{k}(1)&=\frac{\Gamma(N/2)}{2\pi N^{N/2-1}}u_k\, \left\{\int_{J_N}
\frac{1}{2(\gamma_*+iz-\Lambda_k)}\exp\big(NG(\gamma_*+iz)\big)\de z +O(e^{NG(\gamma_*)-c(\log N)^2})\right\}\\
&=\frac{\Gamma(N/2)}{2\pi N^{N/2-1}}u_k\,e^{NG(\gamma_*)} \left\{
\sqrt{\frac{2\pi}{NG^{(2)}(\gamma_{*})}}\frac{1}{2(\gamma_*-\Lambda_k)} +O(N^{-1+\eps})\right\}\, .
\end{align*}
Comparing with the above, we get
\[
E_{k}(1)=(1+O(N^{-c}))\frac{u_{k}}{2(\gamma_{*}-\Lambda_{k})}\int e^{\langle\bsig,\bA\bsig\rangle+\< \bu,\bsig\>}
\mu_0(\de\bsig).
\]
This gives (\ref{eq:laplace-mean}).
\end{proof}

\begin{lem}\label{lem:laplace-moments}
    Let $\bA\in\bbR^{N\times N}$ be a GOE matrix scaled by $\alpha < 1/2$. Assume that $\bu\in \bbR^N$ is such that $\|\bu\|\le N^{c_0}$. For $\ell\in [L]$, consider a collection of pairs of indices $(i_\ell,j_\ell)$ with $i_\ell\ne j_\ell\in [2k]$ and integers $r_\ell \ge 3$. Let $R = \sum_{\ell=1}^{L} r_\ell$. We have that the following claim holds with probability at least $1-\exp(-cN)$:

    Uniformly in $h\in [N]$,
    \begin{align}
        &\frac{\int \prod_{i=1}^{2k}\bsig_h^i\prod_{\ell=1}^{L}\<\bsig^{i_\ell},\bsig^{j_{\ell}}\>^{r_\ell}\exp\lt(\sum_{i=1}^{2k}\<\bu,\bsig^i\>+\<\bsig^i,\bLambda\bsig^i\>\rt)\mu_0^{\otimes 2k}(\de\bsig)}{\int \exp\lt(\sum_{i=1}^{2k}\<\bu,\bsig^i\>+\<\bsig^i,\bLambda\bsig^i\>\rt)\mu_0^{\otimes 2k}(\de\bsig)} \nonumber\\
        &=O_{k,L,R}\lt(\max_{a\le \min(k,L)}|u_h|^{2k-2a} N^{(R-a)/2} (1+\|\bu\|)^{2R}\rt). \label{eq:laplace-moments}
    \end{align}
\end{lem}
\begin{proof}
    As before, we perform a change of basis and assume that $\bA=\bLambda$ is diagonal.
    Consider
    \begin{align*}
        &E(z_1,\dots,z_{2k}) \\
        &:= (\prod_{i=1}^{2k}z_i)^{N/2-1}\int \prod_{i=1}^{2k}(\bsig_h^i\sqrt{z_i})\prod_{\ell=1}^{L}\<\bsig^{i_\ell}\sqrt{z_{i_\ell}},\bsig^{j_{\ell}}\sqrt{z_{j_\ell}}\>^{r_\ell}\exp\lt(\sum_{i=1}^{2k}\<\bu,\bsig^i\>\sqrt{z_i}+\<\bsig^i,\bLambda\bsig^i\>z_i\rt)\mu_0^{\otimes 2k}(\de\bsig).
    \end{align*}
    Then the multivariate Laplace transform of $E$ is
    \begin{align*}
        F(N(t_1,\dots,t_{2k})) &= \frac{\Gamma(N/2)^{2k}}{(N\pi)^{kN}} \int_{\bbR^N} \prod_{i=1}^{2k}y^i_h \prod_{\ell=1}^{L}\<\by^{i_\ell},\by^{j_\ell}\>^{r_\ell} \exp\lt(-\sum_{i=1}^{2k}\lt(t_i\|\by^i\|^2 - \<\bu,\by^i\> - \<\bLambda, (\by^i)^{\otimes2}\>\rt)\rt) \de\by\\
        &= \frac{\Gamma(N/2)^{2k}}{N^{kN}}\exp\lt\{-\sum_{i=1}^{2k}\lt(\frac{1}{2}\sum_{h'=1}^{N}\log(t_i-\Lambda_{h'})+\sum_{h'=1}^{N}\frac{u_{h'}^2}{4(t_i - \Lambda_{h'})}\rt)\rt\} \cdot \frak{G},
    \end{align*}
    where, for
    $\by^i = \frac{1}{2}(t_i-\bLambda)^{-1}\bu + \bw^i$ with $\bw^i$ independently distributed according to $\cN\lt(0,\frac{1}{2}(t_i-\bLambda)^{-1}\rt)$,
    \begin{align}
        \frak{G} &:= \bbE \lt[\prod_{i=1}^{2k}y^i_h \prod_{\ell=1}^{L}\<\by^{i_\ell},\by^{j_\ell}\>^{r_\ell}\rt] \nonumber\\
        &= \bbE \bigg[\prod_{i=1}^{2k}\lt(\frac{u_h}{2(t_i-\Lambda_h)}+w^i_h\rt) \nonumber\\
        &\,\, \prod_{\ell=1}^{L}\lt(\<\bw^{i_\ell},\bw^{j_\ell}\>+\frac{1}{2}\<\bw^{i_\ell},(t_{j_\ell}-\bLambda)^{-1}\bu\>+\frac{1}{2}\<\bw^{j_\ell},(t_{i_\ell}-\bLambda)^{-1}\bu\>+\frac{1}{4}\<(t_{j_\ell}-\bLambda)^{-1}\bu,(t_{i_\ell}-\bLambda)^{-1}\bu\>\rt)^{r_\ell}\bigg]. \label{eq:moment-G}
    \end{align}
    Assume that $\min_{i \in [2k]}\Re(t_i - \max_{h'} \Lambda_{h'}) > c > 0$, and recall that $\|\bu\|\le N^{c_0}$. Define $\cR$ to be a tuple of sets of length $2k+R$, where, for $1\le a\le 2k$, $\cR_a$ is a subset of $\{a\}$, and for $a>2k$, $\cR_a$ is a subset of the pair of indices $\{i_\ell,j_\ell\}$ in the corresponding term. As such, each tuple $\cR$ represents a term in the expansion of (\ref{eq:moment-G}). If there exists an index in $[2k]$ that appears an odd number of times among the sets in $\cR$, then the contribution of the corresponding term to (\ref{eq:moment-G}) is $0$. Consider the tuples $\cR$ where each index appears an even number of times. Let $B(\cR)$ be the collection of indices $a\le 2k$ where $\cR_a = \{a\}$, and let $b(\cR)=|B(\cR)|$. The indices $B(\cR)$ must appear an odd number of times among the remaining sets $(\cR_j)_{j=2k+1}^{2k+R}$. In each possible way to pick out terms among $(\cR_j)_{j=2k+1}^{2k+R}$ so that each index in $B(\cR)$ appears at least once, let $d(\cR)$ denote the number of sets with $|\cR_j|=1$ among these terms.
    Among the remaining terms, each of the indices in $B(\cR)$ not covered by the $d(\cR)$ sets can be matched to terms among $(\cR_j)_{j=2k+1}^{2k+R}$. Consider an arbitrary way to pair up all remaining indices appearing in the terms into pairs; let $f(\cR) \le R - d(\cR) - (b(\cR)-d(\cR))/2$ be the number of such pairs. For each such term and fixed pairing, we can upper bound its contribution to (\ref{eq:moment-G}) by
    \[
        O\lt(|u_h|^{2k-b(\cR)+d(\cR)}N^{f(\cR)/2} (1+\|\bu\|^2)^{R}\rt) \le O\lt(\max_{a\le \min(k,L)}|u_h|^{2k-2a} N^{(R-a)/2} (1+\|\bu\|)^{2R}\rt),
    \]
    noting the constraints $0\le d(\cR) \le b(\cR)$ and $f(\cR) \le R - d(\cR) - (b(\cR)-d(\cR))/2$.
    Thus, we have
    \begin{align*}
        \frak{G} = O_{k,L,R}\lt(\max_{a\le \min(k,L)}|u_h|^{2k-2a} N^{(R-a)/2} (1+\|\bu\|)^{2R}\rt).
    \end{align*}
    Hence, for $R=\sum_\ell r_\ell \ge 3L$,
    \begin{align*}
        \frak{G} = O_{k,L,R}\lt(\max_{a\le \min(k,L)}|u_h|^{2k-2a} N^{(R-a)/2} (1+\|\bu\|)^{2R}\rt).
    \end{align*}
    Taking the inverse Laplace transform and integrating on $t_i = \gamma_* + ix_i$, for $\gamma_*$ defined in Eq.~(\ref{eq:GammaDef}), noting that, as in Lemma \ref{lem:laplace}, we can restrict the integration to the range $x_i \in [-\ell_N,\ell_N]$ for $\ell_N = (\log N)/\sqrt{N}$, we obtain that
    \begin{align*}
        &\frac{\int \prod_{i=1}^{2k}\bsig_h^i\prod_{\ell=1}^{L}\<\bsig^{i_\ell},\bsig^{j_{\ell}}\>^{r_\ell}\exp\lt(\sum_{i=1}^{2k}\<\bu,\bsig^i\>+\<\bsig^i,\bLambda\bsig^i\>\rt)\mu_0^{\otimes 2k}(\de\bsig)}{\int \exp\lt(\sum_{i=1}^{2k}\<\bu,\bsig^i\>+\<\bsig^i,\bLambda\bsig^i\>\rt)\mu_0^{\otimes 2k}(\de\bsig)} \\
        &=O_{k,L,R}\lt(\max_{a\le \min(k,L)}|u_h|^{2k-2a} N^{(R-a)/2} (1+\|\bu\|)^{2R}\rt).
    \end{align*}
\end{proof}

The next lemma states that, under a purely quadratic Hamiltonian, and
for small field, the overlap concentrates near zero.
\begin{lem}[Overlap concentration in quadratic models]\label{lem:overlap-2}
Define
\begin{align*}
A_2(t)^c:=\{(\bsig^1,\bsig^2)\in S_N\times S_N: |\<\bsig^1,\bsig^2\>_N|\ge t\}\, .
\end{align*}
Assuming that $\|\bu\|^2 \le \delta N$ for $\delta$ sufficiently small, we have for some constant $c,C>0$ that, with probability $1-e^{-cN}$,
\begin{align}\label{eq:tail}
\frac{\int_{A_2(t)^c}\exp(H_{\le2}(\bsig^{1})+H_{\le2}(\bsig^{2}))
\mu_0^{\otimes 2}(\de\bsig)}{\int \exp(H_{\le2}(\bsig^{1})+H_{\le2}(\bsig^{2}))\mu^{\otimes 2}_0(\de\bsig)}\le
\exp\Big\{-cN\big(t-C\|\bu\|_N\big)_+^2\Big\}.
\end{align}
The same bound holds for two replicas with different fields and quadratic matrices in a sufficiently small operator-norm neighborhood of $\bA$, replacing $\|\bu\|_N^2$ by $\max_i\|\bu_i\|_N^2$.
\end{lem}
\begin{proof}
Consider the Hamiltonian $H(\bsig^{1},\bsig^{2})=H_{\le2}(\bsig^{1})+H_{\le2}(\bsig^{2})+2\theta\langle\bsig^{1},\bsig^{2}\rangle$. Let $\Lambda_{i}$ be the eigenvalues of the
quadratic component $\bA$ of $H$. Using the Laplace transform as in Lemma
\ref{lem:laplace},
\begin{align*}
 & \int \exp\Big(\langle \bu,\bsig^{1}+\bsig^{2}\rangle+\langle \bA,(\bsig^1)^{\otimes 2}+(\bsig^2)^{\otimes 2}\rangle\Big)\exp\left(2\theta\langle\bsig^{1},\bsig^{2}\rangle\right)\mu_0^{\otimes 2}(\de\bsig)\\
 & = \left(\frac{\Gamma(N/2)}{(2\pi)N^{N/2-1}}\right)^2 \int_{-\infty}^{\infty}\exp\Big(2N\tilde{G}_{\theta}(\gamma_1+iz_1, \gamma_2+iz_2)\Big)\de z_1\de z_2,
\end{align*}
where
\[
\tilde{G}_{\theta}(z_1,z_2)=\frac{z_1+z_2}{2}-\frac{1}{4N}\sum_{i=1}^{N}\log((z_1-\Lambda_{i})(z_2-\Lambda_i)-\theta^2)+\frac{1}{8N}\sum_{i=1}^{N}\frac{u_{i}^{2}(z_1+z_2-2\Lambda_i + 2\theta)}{(z_1-\Lambda_i)(z_2-\Lambda_i)-\theta^2}.
\]

We also denote
\[
G_\theta(z) = \tilde G_\theta(z,z) = z-\frac{1}{4N}\sum_{i=1}^{N}\log((z-\Lambda_{i})^2-\theta^2)+\frac{1}{4N}\sum_{i=1}^{N}\frac{u_{i}^{2}(z-\Lambda_i + \theta)}{(z-\Lambda_i)^2-\theta^2}.
\]
Let $\gamma_{*}(\theta)$ be a stationary point of $G_\theta$ on $\mathbb{R}$
so that $\gamma_{*}(\theta)>\max\Lambda_{i}+|\theta|$ (there exists
a unique such point), and $\gamma_* = \gamma_*(0)$.

As in Lemma \ref{lem:laplace}, note that
\begin{align}
&\exp\left(2\Re\left(\log((\gamma_*(\theta)-\Lambda_i + iz_1)(\gamma_*(\theta)-\Lambda_i + iz_2)-\theta^2) - \log((\gamma_*(\theta)-\Lambda_i)^2-\theta^2)\right)\right) \nonumber\\
&= \left(\frac{(\gamma_*(\theta)-\Lambda_i)^2-\theta^2-z_1z_2}{(\gamma_*(\theta)-\Lambda_i)^2-\theta^2}\right)^2 + \left(\frac{(\gamma_*(\theta)-\Lambda_i)(z_1+z_2)}{(\gamma_*(\theta)-\Lambda_i)^2-\theta^2}\right)^2 \nonumber\\
&= \frac{((\gamma_*(\theta)-\Lambda_i)^2-\theta^2)^2 + (z_1z_2)^2 + 2\theta^2 z_1z_2 + (z_1^2+z_2^2)(\gamma_* (\theta)-\Lambda_i)^2}{((\gamma_*(\theta)-\Lambda_i)^2-\theta^2)^2} \nonumber\\
&\ge 1 + \frac{z_1^2+z_2^2}{(\gamma_*(\theta)-\Lambda_i)^2-\theta^2}. \label{eq:log-est}
\end{align}
Furthermore,
\begin{align}
    &\Re\left(\frac{(\gamma_*(\theta)-\Lambda_i + \theta) + i(z_1+z_2)/2}{N((\gamma_*(\theta)-\Lambda_i)^2-\theta^2-z_1z_2 + i (\gamma_*(\theta)-\Lambda_i)(z_1+z_2))}-\frac{1}{N(\gamma_*(\theta)-\Lambda_i-\theta)}\right) \nonumber\\
    &= \frac{-(z_1z_2)^2-\theta (\gamma_*(\theta)-\Lambda_i+\theta)(z_1+z_2)^2/2-((\gamma_*(\theta)-\Lambda_i)^2-\theta^2)(z_1^2 + z_2^2)/2}{N(\gamma_*(\theta)-\Lambda_i-\theta)(((\gamma_*(\theta)-\Lambda_i)^2 - \theta^2 - z_1z_2)^2 + ((\gamma_*(\theta)-\Lambda_i)(z_1+z_2))^2)} \nonumber \\
    &\le 0. \label{eq:u-est}
\end{align}
Given (\ref{eq:log-est}) and (\ref{eq:u-est}), we can proceed as in Lemma \ref{lem:laplace} to restrict the integral over $z_1$ and $z_2$ to the range $|z_1|,|z_2|<(\log N)/\sqrt{N}$, incurring an error $e^{-c(\log N)^2}$. Then by similarly expanding around $(\gamma_*(\theta),\gamma_*(\theta))$, we obtain
\begin{align*}
 & \int \exp\Big(\langle \bu,\bsig^{1}+\bsig^{2}\rangle+\langle \bA,(\bsig^1)^{\otimes 2}+(\bsig^2)^{\otimes 2}\rangle\Big)\exp\left(2\theta\langle\bsig^{1},\bsig^{2}\rangle\right)\mu_0^{\otimes 2}(\de\bsig)\\
 & = \left(\frac{\Gamma(N/2)}{(2\pi)N^{N/2-1}}\right)^2  e^{2NG_{\theta}(\gamma_*(\theta))} \lt\{\frac{\pi}{N \det(\nabla^2 \tilde{G}_\theta(\gamma_*(\theta),\gamma_*(\theta)))^{1/2}} + O(N^{-3/2+\epsilon})\rt\}.
\end{align*}

When $\|\bu\|^{2} \le \delta N$, we have for $G$ as in (\ref{eq:GDef}) that
\[
G_{\theta}(\gamma_{*}(\theta))-G(\gamma_{*})=O(\theta^{2})+O(\theta \|\bu\|^2/N).
\]
On the other hand, by Lemma \ref{lem:laplace},
\begin{align*}
&\int \exp(\langle \bu,\bsig^{1}+\bsig^{2}\rangle+\langle \bA,(\bsig^1)^{\otimes 2}+(\bsig^2)^{\otimes 2}\rangle)\mu_0^{\otimes 2}(\de\bsig)\\
&=(1+O(N^{-c})) \left(\frac{\Gamma(N/2)}{(2\pi)N^{N/2-1}}\right)^2 e^{2NG(\gamma_*)}\left(\sqrt{\frac{2\pi}{NG''(\gamma_*)}} + O(N^{-1+\eps})\right)^2.
\end{align*}
In particular,
\begin{align*}
&\frac{\int_{t\le|\langle\bsig^{1},\bsig^{2}\rangle_N|}\exp(H_{\le2}(\bsig^{1})+H_{\le2}(\bsig^{2}))\mu_0^{\otimes 2}(\de\bsig)}{\int_{\bsig^{1},\bsig^{2}}\exp(H_{\le2}(\bsig^{1})+H_{\le2}(\bsig^{2}))\mu_0^{\otimes 2}(\de\bsig)} \\
&\le 2\exp(-2N\theta t) \max_{\epsilon \in \{\pm 1\}}\frac{\int_{\bsig^{1},\bsig^{2}}\exp(\langle \bu,\bsig^{1}+\bsig^{2}\rangle + N\langle \bA,(\bsig^{1})^{\otimes 2}+(\bsig^{2})^{\otimes 2}\rangle)\exp\left(2\epsilon \theta\langle\bsig^{1},\bsig^{2}\rangle\right)\mu_0^{\otimes 2}(\de\bsig)}{\int_{\bsig^{1},\bsig^{2}}\exp(\langle \bu,\bsig^{1}+\bsig^{2}\rangle + N\langle \bA,(\bsig^{1})^{\otimes 2}+(\bsig^{2})^{\otimes 2}\rangle)\mu_0^{\otimes 2}(\de\bsig)} \\
&\le C\exp\Big(-2N\theta t + O(N\theta^2) + O(\theta \|\bu\|^2)\Big).
\end{align*}
Optimizing over $\theta$, we obtain Eq.~\eqref{eq:tail}. For different matrices $\bA_1,\bA_2$ and fields $\bu_1,\bu_2$, the same calculation uses the block matrix $\bigl(\begin{smallmatrix}z_1\bI-\bA_1&-\theta\bI\\-\theta\bI&z_2\bI-\bA_2\end{smallmatrix}\bigr)$. The saddle-point gap gives the same bound $\log\bbE\exp(2\theta\langle\bsig^1,\bsig^2\rangle)\le O(1)+O(N\theta^2)+O(|\theta|\max_i\|\bu_i\|^2)$ uniformly under these perturbations and rescalings.
\end{proof}

We will also need the following lemma, giving an accurate expansion of moments of overlaps in perturbations of quadratic Hamiltonians.

\begin{lem}\label{lem:laplace-moments-2}
    Let $\bA\in\bbR^{N\times N}$ be a GOE matrix scaled by $0\le\alpha<1/2$ with eigenvalues given by $\bLambda$. Let $\bDel\in\bbR^{N\times N}$ be an independent GOE matrix scaled by $0\le\beta\le B$, and assume that $\zeta_1,\zeta_2\in\bbR$ satisfy $|\zeta_1|,|\zeta_2|\le C(\log N)/\sqrt N$. For $i=1,2$, let $\tilde{\bLambda}_i=\bLambda+\zeta_i\bDel$. Assume that $\bu\in\bbR^N$ is such that $\|\bu\|\le N^{c_0}$. Let $r\ge0$ and $L\ge1$ be integers, and set $J=L+r+2$. For $c_0>0$ sufficiently small, we have that the following claim holds with probability at least $1-\exp(-cN)$: There exist $C_{i,j}=O_{r,L}\lt((1+\|\bu\|)^{2r}N^{\lfloor r/2\rfloor}\rt)$ for $i+j\le J$ such that
    \begin{align}
        F_r(\zeta_1,\zeta_2)&:=\frac{\int\<\bsig^1,\bsig^2\>^r\exp\lt(\sum_{i=1}^2\big(\<\bu,\bsig^i\>+\<\bsig^i,\tilde{\bLambda}_i\bsig^i\>\big)\rt)\mu_0^{\otimes2}(\de\bsig)}{\int\exp\lt(\sum_{i=1}^2\big(\<\bu,\bsig^i\>+\<\bsig^i,\tilde{\bLambda}_i\bsig^i\>\big)\rt)\mu_0^{\otimes2}(\de\bsig)}\nonumber\\
        &=C_{0,0}+\sum_{1\le i+j\le J}C_{i,j}\zeta_1^i\zeta_2^j+O_{r,L}(N^{-L/2}).\label{eq:laplace-moments-expanded}
    \end{align}
    Further, for each fixed $i,j\ge0$, uniformly on a sufficiently small fixed neighborhood of $(0,0)$,
    \begin{align}\label{eq:quadratic-moment-derivatives}
        |\partial_{\zeta_1}^i\partial_{\zeta_2}^jF_r(\zeta_1,\zeta_2)|
        \le O_{r,i,j}\lt((1+\|\bu\|)^{2r}N^{\lfloor r/2\rfloor}\rt).
    \end{align}
    The constants are uniform for bounded $B$ and $\alpha$ bounded away from $1/2$. The derivative bound also holds with different fields and symmetric perturbations of bounded operator norm in the two replicas, and on spheres of radii $\sqrt N(1+o(1))$, replacing $\|\bu\|$ by the maximum of the field norms, provided the saddle points remain a fixed distance from the spectra.
\end{lem}
\begin{proof}
    Consider
    \begin{align*}
        E(z_1,z_2)&:=(z_1z_2)^{N/2-1}\int\<\bsig^1\sqrt{z_1},\bsig^2\sqrt{z_2}\>^r\\
        &\qquad\qquad\cdot\exp\lt(\sum_{i=1}^2\big(\<\bu,\bsig^i\>\sqrt{z_i}+\<\bsig^i,\tilde{\bLambda}_i\bsig^i\>z_i\big)\rt)\mu_0^{\otimes2}(\de\bsig).
    \end{align*}
    Then the multivariate Laplace transform of $E$ is
    \begin{align*}
        \widehat E(Nt_1,Nt_2)
        &=\frac{\Gamma(N/2)^2}{(N\pi)^N}\int_{(\bbR^N)^2}\<\by^1,\by^2\>^r
        \exp\lt(-\sum_{i=1}^2\big(t_i\|\by^i\|^2-\<\bu,\by^i\>-\<\by^i,\tilde{\bLambda}_i\by^i\>\big)\rt)\de\by\\
        &=\frac{\Gamma(N/2)^2}{N^N}\exp\lt\{-\frac12\sum_{i=1}^2\log\det(t_i\bI_N-\tilde{\bLambda}_i)
        +\frac14\sum_{i=1}^2\<\bu,(t_i\bI_N-\tilde{\bLambda}_i)^{-1}\bu\>\rt\}\cdot\frak G.
    \end{align*}
    Here, for $\bZ^i=(t_i\bI_N-\tilde{\bLambda}_i)^{-1}$, $\by^i=\frac12\bZ^i\bu+\bw^i$ with $\bw^i$ independently distributed according to $\cN(0,\frac12\bZ^i)$, and
    \begin{align*}
        \frak G(t_1,t_2;\zeta_1,\zeta_2)
        &:=\bbE\lt[\<\by^1,\by^2\>^r\rt]\\
        &=\bbE\bigg[\bigg(\<\bw^1,\bw^2\>+\frac12\<\bw^1,\bZ^2\bu\>
        +\frac12\<\bw^2,\bZ^1\bu\>+\frac14\<\bZ^1\bu,\bZ^2\bu\>\bigg)^r\bigg].
    \end{align*}
    By Wick's formula, $\frak G$ is a rational function of $t_1,t_2$, and hence extends to complex values of $t_i$.

    We work on the event in Lemma \ref{lem:laplace}, intersected with $\{\|\bDel\|_{\op}\le C_B\}$, which has probability at least $1-\exp(-cN)$. Whenever $\Re t_i-\lambda_{\max}(\tilde{\bLambda}_i)\ge d>0$, we have
    \begin{align*}
        \partial_{\zeta_i}^m\bZ^i=m!\bZ^i(\bDel\bZ^i)^m,
        \qquad \|\partial_{\zeta_i}^m\bZ^i\|_{\op}=O_{m,B,d}(1).
    \end{align*}
    Expanding $\frak G$ by Wick's formula, each term is a product of traces and quadratic forms in $\bu$. For odd $r$, every term contains at least one quadratic form in $\bu$. We can thus bound the derivatives by
    \begin{align*}
        |\partial_{\zeta_1}^i\partial_{\zeta_2}^j\frak G|
        &\le O_{r,i+j}\lt(\|\bu\|^{2r}+N^{r/2}\rt),\qquad r\text{ even},\\
        |\partial_{\zeta_1}^i\partial_{\zeta_2}^j\frak G|
        &\le O_{r,i+j}\lt(\|\bu\|^{2r}+\|\bu\|^2N^{(r-1)/2}\rt),\qquad r\text{ odd}.
    \end{align*}
    The same bounds hold after any fixed number of derivatives in $t_1,t_2$.

    Next, for $\tilde G_i(\gamma)=G(\gamma;\tilde{\bLambda}_i,\bu)$, let $\gamma_i=\gamma_*(\tilde{\bLambda}_i,\bu)$. By the resolvent identity and the stationarity equation,
    \begin{align*}
        \tilde G_i'(\gamma)&=1-\frac1{2N}\Tr(\gamma\bI_N-\tilde{\bLambda}_i)^{-1}
        -\frac1{4N}\<\bu,(\gamma\bI_N-\tilde{\bLambda}_i)^{-2}\bu\>,\\
        \tilde G_i''(\gamma_i)&\asymp1,\qquad
        \gamma_i-\gamma_*=O(|\zeta_i|).
    \end{align*}
    Thus, for $|\zeta_i|$ sufficiently small, $\gamma_i$ remains a fixed distance from the spectrum, and its derivatives of every fixed order are bounded.

    We take the inverse Laplace transforms on $t_i=\gamma_i+ix_i$. Set
    \[
        \phi_i(x_i)=\tilde G_i(\gamma_i+ix_i)-\tilde G_i(\gamma_i),
    \]
    and write
    \begin{align*}
        I_r(\zeta_1,\zeta_2)
        &:=\int_{\bbR^2}e^{N(\phi_1(x_1)+\phi_2(x_2))}
        \frak G(\gamma_1+ix_1,\gamma_2+ix_2;\zeta_1,\zeta_2)\de x_1\de x_2.
    \end{align*}
    The numerator and denominator of $F_r$ have the same prefactor and exponential saddle-point factors. Consequently,
    \[
        F_r(\zeta_1,\zeta_2)=\frac{I_r(\zeta_1,\zeta_2)}{I_0(\zeta_1,\zeta_2)},
        \qquad
        I_0=\frac{2\pi}{N\sqrt{\tilde G_1''(\gamma_1)\tilde G_2''(\gamma_2)}}(1+o(1))\asymp N^{-1}.
    \]
    As in Lemma \ref{lem:laplace}, the integrals may be restricted to $|x_i|\le\ell_N=(\log N)/\sqrt N$, with an error bounded by
    $O_{r,i,j}\lt((1+\|\bu\|)^{2r}N^{\lfloor r/2\rfloor}e^{-c(\log N)^2}\rt)$, also after any fixed $i,j$ derivatives in $\zeta_1,\zeta_2$. Indeed, the resolvents remain bounded on these contours, and fixed derivatives of the exponential contribute only polynomial factors in $N$ outside this range.

    On $|x_i|\le\ell_N$, the identity $\tilde G_i'(\gamma_i)=0$ gives, for each fixed $m\ge0$,
    \[
        \partial_{\zeta_i}^m\phi_i(x_i)=O_m(x_i^2),
        \qquad \Re\phi_i(x_i)\le-cx_i^2.
    \]
    Hence differentiating $e^{N\phi_i(x_i)}$ introduces only a polynomial in $Nx_i^2$, which is absorbed by the Gaussian integral. Combining this with the derivative bounds for $\frak G$, we obtain
    \begin{align*}
        |\partial_{\zeta_1}^i\partial_{\zeta_2}^jI_r|
        &\le O_{r,i,j}\lt((1+\|\bu\|)^{2r}N^{\lfloor r/2\rfloor-1}\rt),\\
        |\partial_{\zeta_1}^i\partial_{\zeta_2}^jI_0|&\le O_{i,j}(N^{-1}).
    \end{align*}
    Differentiating $F_r=I_r/I_0$ proves \eqref{eq:quadratic-moment-derivatives}.

    Finally, Taylor's expansion to order $J=L+r+2$ yields
    \begin{align*}
        F_r(\zeta_1,\zeta_2)
        &=F_r(0,0)+\sum_{1\le i+j\le J}C_{i,j}\zeta_1^i\zeta_2^j\\
        &\qquad+O_{r,L}\lt((1+\|\bu\|)^{2r}N^{\lfloor r/2\rfloor}
        (|\zeta_1|+|\zeta_2|)^{J+1}\rt),
    \end{align*}
    where
    \[
        C_{i,j}=\frac{1}{i!j!}\partial_{\zeta_1}^i\partial_{\zeta_2}^jF_r(0,0).
    \]
    Under the assumption $|\zeta_i|\le C(\log N)/\sqrt N$, the remainder is $O_{r,L}(N^{-L/2})$ for $c_0$ sufficiently small, proving \eqref{eq:laplace-moments-expanded}. The same argument applies to different fields, bounded symmetric perturbations, and rescaled radii as in the statement.
\end{proof}

\section{Estimates of restricted partition functions}

In this section we estimate modified partition functions that are obtained
by suitable restrictions of the integral over $\bsig$, always under the assumption
\eqref{eq:replica-symmetry}.
Namely, for any Borel set $U\subseteq (S_N)^{\otimes m}$,
\begin{align}
    Z_m(U):=\int_U e^{\sum_{i=1}^mH(\bsig^i)}\mu_0^{\otimes m}(\de\bsig)\, ,
\end{align}
with subscript omitted if $m=1$. If $U=S_N$, we write simply $Z=Z(S_N)$.
We also denote by $Z_{\le 2,m}(U)$ the same integral in which $H(\bsig)$ is
replaced by $H_{\le 2}(\bsig)$:

\begin{align}
    Z_{\le 2, m}(U):=\int_U e^{\sum_{i=1}^mH_{\le 2}(\bsig^i)}\mu_0^{\otimes m}(\de\bsig)\, .
\end{align}
We will occasionally omit the subscript $m$ when the dimension of $U$ is clear from the context.

Throughout this section, we follow the notations
$\<\bx,\by\>_N = \langle \bx,\by\rangle/N$, so $\langle \bx,\bx\rangle_N = \|\bx\|_N^2$.

As for the restrictions, an important role is played by the typical set:
\begin{align}
        T(\delta)=\left\{\bsig\in S_N:\int_{\bsig':|\<\bsig',\bsig\>_N|>\delta}
        e^{H(\bsig')}\mu_0(\de\bsig')<e^{-c_1(\delta)N}\min\lt(\int e^{H(\bsig)}\mu_0(\de\bsig); e^{N\xi(1)/2}\rt)\right\}\, .\label{eq:TypicalDef}
    \end{align}
We further define $A_{m}(\delta)\subseteq (S_N)^m$ to be the set of $m$-tuples
of vectors which are nearly orthogonal. Namely:
\begin{align}\label{eq:Am_def}
A_{m}(\delta):=\Big\{(\bsig^i)_{i\le m}:\, \bsig^i\in S_N, |\<\bsig^i,\bsig^j\>_N|\le \delta\,\,\forall i\neq j\Big\}\, .
\end{align}
Finally, we consider the set of $m$-tuples in $T=T(\delta)$ that are nearly orthogonal:
\begin{align}\label{eq:AmT_def}
A_{m}(T,\delta):=\Big\{(\bsig^i)_{i\le m}:\, \bsig^i\in T, |\<\bsig^i,\bsig^j\>_N|\le \delta\,\,\forall i\neq j\Big\}\, .
\end{align}
In particular $A_{m}(T,\delta) = T^{m}\cap A_m(\delta)$.

Our first lemma establishes that, under the Gibbs measure, non-typical points are exponentially rare.
\begin{lem}[Most points are typical]\label{lem:truncate}
    For any $\delta > 0$, there exist $u(\delta), c_1(\delta),c_2(\delta) > 0$ such that the following holds.
    Let $H(\bsig)$ be defined as per Eq.~\eqref{eq:FirstLocal} and suppose $\norm{\bu} \le u(\delta) \sqrt{N}$.
    Let $T(\delta)$ be defined as per Eq.~\eqref{eq:TypicalDef}.

    Then, with probability at least $1-\exp(-c_2(\delta)N)$,
    \begin{align}\label{eq:truncate}
        Z(T(\delta))\ge(1-e^{-Nc_2(\delta)})\cdot Z\,  .
    \end{align}
    Furthermore, there is $c_3(\delta)>0$ such that, with probability at least $1-\exp(-c_3(\delta)N)$,
    \begin{align}\label{eq:truncate-quad}
        Z_{\le 2}(T(\delta)^c) \le e^{-c_3(\delta)N} Z_{\le 2}.
    \end{align}
   Finally
   \beq
        \label{eq:atypical-small-in-expectation}
        \EE \lt[\int_{T(\delta)^c} e^{H_{\ge 2}(\bsig)} ~\mu_0(\de \bsig)\rt]
        \le e^{-c_1(\delta) N} \EE Z_{\ge 2}.
    \eeq
\end{lem}
\begin{proof}
    The second inequality in \eqref{eq:replica-symmetry} is
    termed ``strictly RS'' in \cite{huang2023constructive}; see Eq. (2.7) therein.
    By Proposition 3.1 of that paper,
    \[
        \EE \int_{T(\delta)} e^{H_{\ge 2}(\bsig)} ~\mu_0(\de \bsig)
        \ge (1-e^{-c_1(\delta) N}) e^{N\xi(1)/2}.
    \]
    (While this proposition states a bound of $(1-o(1)) \exp(N\xi(1)/2)$, its proof shows the $1-o(1)$ is in fact $1-e^{-c_1(\delta) N}$.)
    As $\EE Z_{\ge 2} = \exp(N\xi(1)/2)$, for $Z_{\ge 2} := \int_{S_N} \exp H_{\ge 2}(\bsig) ~\mu_0(\de \bsig)$,
    this implies Eq.~\eqref{eq:atypical-small-in-expectation}.

    By Markov's inequality, with probability $1-e^{-c_1(\delta) N/5}$,
    \[
        \int_{T(\delta)^c} e^{H_{\ge 2}(\bsig)} ~\mu_0(\de \bsig)
        \le e^{-4c_1(\delta) N/5} \EE Z_{\ge 2}.
    \]
    By \cite[Proposition 2.3]{talagrand2006spherical}, \eqref{eq:replica-symmetry} implies that $\fr1N \log Z_{\ge 2} \to_p \xi(1)/2$.
    By standard concentration properties of $\fr1N \log Z_{\ge 2}$, with probability $1-e^{-c_2(\delta) N}$,
    \[
        Z_{\ge 2} \ge e^{-c_1(\delta) N/5} \EE Z_{\ge 2}.
    \]
    On the intersection of these events,
    \[
        \int_{T(\delta)^c} e^{H_{\ge 2}(\bsig)} ~\mu_0(\de \bsig)
        \le e^{-3c_1(\delta) N/5} Z_{\ge 2}.
    \]
    Finally, set $u(\delta) = c_1(\delta) / 5$, so that for all $\bsig \in S_N$,
    \[
        |H(\bsig) - H_{\ge 2}(\bsig)| = |\la \bu, \bsig \ra| \le c_1(\delta) N/5.
    \]
    Thus
    \[
        \int_{T(\delta)^c} e^{H(\bsig)} ~\mu_0(\de \bsig)
        \le e^{-c_1(\delta) N/5}
        \int_{S_N} e^{H(\bsig)} ~\mu_0(\de \bsig).
    \]
    The conclusion (\ref{eq:truncate}) follows with $c(\delta) = \min(c_1(\delta)/6, c_2(\delta)/2)$.

    The same planted estimate, tilting only by $e^{H_2(\bsig)}$, gives $\EE Z_{\le2}(T(\delta)^c)\le e^{-cN+N\xi_{\le2}(1)/2}$ after decreasing $u(\delta)$. Finally, from Markov's inequality, we have with probability $1-e^{-c_3(\delta)N}$ that $$Z_{\le 2}(T(\delta)^c) \le e^{-c_3(\delta)N} e^{N\xi_{\le2}(1)/2}.$$ Then (\ref{eq:truncate-quad}) follows from standard concentration properties.
\end{proof}

The next lemma states that we can anneal over terms
of degree higher than $2$ in the Hamiltonian. This will be the most important
technical result of the section.
\begin{lem}
\label{lem:high-acc-anneal}
Let $H(\bsig)$ be defined as per Eq.~\eqref{eq:FirstLocal} and define
 $T=T(\delta)$ as in Eq.~\eqref{eq:TypicalDef}. Assume that $\|\bu\| \le N^{c_0}$ for $c_0$ sufficiently small given $\xi, L, k$.
Under assumption \eqref{eq:replica-symmetry},
for all $L,k>0$
and $\eps>0$, there exists $C=C(L,k)>0$ such that
the following holds with probability at least $1-\exp(-N/C)$
\begin{align}
\bbE_{\ge3}\Big\{\big(Z(T)-\bbE_{\ge3}Z(T)\big)^{2k}\Big\}
\le C\, N^{-L/2}\, \big(\bbE_{\ge3}Z(T)\big)^{2k}\, ,
\label{eq:ZT_MomentBound}
\end{align}
and further
\begin{align}\label{eq:HighAcc-PartitionFun}
\PP\Big\{\big| Z-\bbE_{\ge3}Z(T) \big|>\eps\, \bbE_{\ge3}Z(T)\Big\} \le C \eps^{-2k}N^{-L/2}  + e^{-N/C}.
\end{align}
We also have, with probability at least $1-\exp(-N/C)$,
\begin{align}\label{eq:Z3T}
    \bbE_{\ge3} Z(T) = (1+O(e^{-N/C})) \bbE_{\ge3} Z.
\end{align}
Further, letting $(\bv_k)_{k\le N}$ be the basis of
eigenvectors of $\bW_2$, for each $i\in [N]$,
we have
\begin{align}\label{eq:Magn-No-Field}
&\PP\left\{\frac{\left|\int_T\langle\bv_i,\bsig\rangle e^{H(\bsig)}\mu_0(\de\bsig)\right|}{\bbE_{\ge3}Z}
\ge N^\eps(1+\|\bu\|)^{Ck}(|\langle\bv_i,\bu\rangle|+CN^{-1/2})\right\}\nonumber\\
&\hspace{3cm}\le C(N^{-2\eps k}+e^{-N/C}).
\end{align}
\end{lem}
Before proving this lemma, we state and prove a number of key estimates.

Our first lemma establishes that (in expectation) the partition function
in $A_{2k}(\delta)$ is dominated by the subset $A_{2k}(\delta,T)$.
\begin{lem}[Orthogonal frames are mostly typical]\label{lem:key-error}
    Define
 $T=T(\delta)$ as in Eq.~\eqref{eq:TypicalDef}. We have for $\delta>0$ sufficiently small and appropriate $c, c'>0$ that, if $\|\bu\| \le c'\sqrt{N}$,
    \begin{align}\label{eq:key-error}
    \bbE\, Z_{2k}\big(\{(\bsig^i)_{i\le 2k}\in A_{2k}(\delta): \bsig^1\in T^c\big\} \big)
    \le e^{-cN}  \bbE\, Z_{2k}\big(A_{2k}(\delta)\big).
\end{align}
As a consequence,
 \begin{align}\label{eq:key-error-2}
    \bbE\, Z_{2k}\big( A_{2k}(\delta,T) \big)
    \ge (1-e^{-cN})  \bbE\, Z_{2k}\big(A_{2k}(\delta)\big).
\end{align}
\end{lem}
\begin{proof}
    We have
    \begin{align*}
        \bbE\Big\{H_{\ge2}(\brho) \Big| \sum_{i=1}^{2k}H_{\ge2}(\bsig^i)\Big\} &= \frac{\bbE H_{\ge2}(\brho)\sum_{i=1}^{2k}H_{\ge2}(\bsig^i)}{\bbE (\sum_{i=1}^{2k}H_{\ge2}(\bsig^i))^2} \sum_{i=1}^{2k} H_{\ge2}(\bsig^i)\\
        &= \frac{\sum_{i=1}^{2k}\xi(\langle \brho,\bsig^i\rangle_N)}{\sum_{i,j\in [2k]}\xi(\langle \bsig^i,\bsig^j\rangle_N)} \sum_{i=1}^{2k} H_{\ge2}(\bsig^i),
    \end{align*}
    and for $\widehat{H}(\brho) = H_{\ge2}(\rho)-\bbE[H_{\ge2}(\brho) | \sum_{i=1}^{2k}H_{\ge2}(\bsig^i)]$,
    \begin{align*}
        \frac{1}{N}\bbE\big[\widehat{H}(\brho^1)\widehat{H}(\brho^2)\big] = \xi(\langle \brho^1,\brho^2\rangle_N) - \frac{(\sum_{i=1}^{2k} \xi(\langle\brho^1,\bsig^i\rangle_N)(\sum_{i=1}^{2k} \xi(\langle\brho^2,\bsig^i\rangle_N)}{\sum_{i,j\in [2k]}\xi(\langle \bsig^i,\bsig^j\rangle_N)}.
    \end{align*}

    For each $|q_1| \ge \delta$ and $q_2,\dots,q_{2k}\in [0,1]$, consider the band $\Band_*(\{\bsig^i\})$
    of vectors $\brho$ with $\langle \brho,\bsig^i\rangle_N = q_i$ for all $i\in [2k]$. Write $\brho = \bx + \sqrt{1-\tilde{q}^2} \btau$ where $\bx \in \mathrm{span}(\bsig^1,\dots,\bsig^{2k})$ and $\|\btau\|^2=N$, $\btau\perp \mathrm{span}(\bsig^1,\dots,\bsig^{2k})$. Define the process $\overline{H}(\btau) = \widehat{H}(\brho)$, which is a $p$-spin model with corresponding mixture $\txi(t) = \txi(t;\bq,(\bsig^i)_{i=1}^{2k})$
    given by
     \begin{align*}
\txi(t;\bq,(\bsig^i)_{i=1}^{2k}) = \xi(\tilde{q}^2+(1-\tilde{q}^2)t\big) -\frac{\big(\sum_{i=1}^{2k}
\xi(q_i)\big)^2}{\sum_{i,j\in [2k]}\xi(\langle \bsig^i,\bsig^j\rangle_N)}\, .
     \end{align*}

We define the free energy
  \begin{align*}
  \Phi(\bq; (\bsig^i)_{i=1}^{2k}) := \frac{1}{N}\log\int_{\Band_*(\{\bsig^i\})} e^{H_{\ge2}(\brho)}\mu_0(\de\brho)\, .
  \end{align*}
 Following the proof of Lemma 3.3 of \cite{huang2023constructive}, we see that the replica-symmetric bound implies
 that the following holds with high probability:
    \begin{align}\label{eq:replica-sym}
        \Phi(\bq; (\bsig^i)_{i=1}^{2k}) \le \frac{\sum_{i=1}^{2k}\xi(q_i)}{\sum_{i,j\in [2k]}\xi(\langle \bsig^i,\bsig^j\rangle_N)} \frac{1}{N}\sum_{i=1}^{2k} H_{\ge2}(\bsig^i) + \frac{1}{2}\xi(1) - \frac{1}{2}\xi(\tilde{q}) +\frac{1}{2}\tilde{q} + \frac{1}{2} \log(1-\tilde{q}) + o_N(1).
    \end{align}
    Let $\bG=(\langle\bsig^i,\bsig^j\rangle_N)_{i,j\le2k}$.
The generalized Bessel inequality, applied to the operator with columns
$\bsig^i/\sqrt N$, gives
\begin{align*}
 \sum_{i=1}^{2k}\langle\bx,\bsig^i\rangle_N^2
 &\le \lambda_{\max}(\bG)\|\bx\|_N^2
 \le \bigl(1+(2k-1)\delta\bigr)\|\bx\|_N^2.
\end{align*}
Here the last inequality follows from the unit diagonal and the bound
$|G_{ij}|\le\delta$ for $i\ne j$. Thus, with $D=1+(2k-1)\delta$,
\[
 \tilde q^2=\|\bx\|_N^2\ge D^{-1}\sum_{i=1}^{2k}q_i^2.
\]
Since $q_i=\langle\bx,\bsig^i\rangle_N$, we also have
$|q_i|\le\|\bx\|_N=\tilde q$ for each $i$. Hence, for every $p\ge2$,
\[
 \sum_{i=1}^{2k}|q_i|^p
 \le\tilde q^{p-2}\sum_iq_i^2\le D\tilde q^p.
\]
Summing against the nonnegative mixture coefficients gives
\[
 \sum_{i=1}^{2k}\xi(q_i)
 \le\sum_i\xi(|q_i|)\le D\xi(\tilde q)
 \le\xi(\tilde q)+(2k-1)\delta\,\xi(1)\tilde q^2.
\]
Thus the error is $C_k\delta\tilde q^2$, uniformly under
\eqref{eq:uniform-mixtures}, without evaluating $\xi$ beyond $1$.
Moreover, strict
replica symmetry and $\xi''(0)<1$ imply, for some $\kappa>0$,
\[
 \xi(u)+u+\log(1-u)\le-\kappa u^2\qquad(0<u<1).
\]
Let $U=\sum_{i,j\le2k}\xi(\langle\bsig^i,\bsig^j\rangle_N)$.
At $E=U$, substituting the preceding Bessel estimate into
\eqref{eq:replica-sym} therefore gives
\[
 \Phi(\bq;(\bsig^i)_{i\le2k})
 \le\frac12\xi(1)-\frac\kappa2\tilde q^2
       +C_k\delta\tilde q^2+o_N(1).
\]
Choose $\delta$ so small that $C_k\delta\le\kappa/4$.
If $|q_1|\ge\delta$, then $\tilde q\ge\delta/\sqrt D$.
The perturbation $|E-U|\le\eta\sqrt U$ costs at most $C_k\eta$:
indeed, $|\xi(r)|\le r^2\xi(1)$ for $|r|\le1$, so for small
$\delta$ we have $U\ge k\xi(1)$ and
$\sum_i|\xi(q_i)|/\sqrt U\le2\sqrt k\sqrt{\xi(1)}$.
This bound is uniform under \eqref{eq:uniform-mixtures}; the zero
mixture is immediate without Gaussian conditioning.
Taking $\eta$ sufficiently small compared with $\kappa\delta^2/D$
thus preserves the strictly negative gap used below.
With $U$, $\delta$, and $\eta$ as above, given $\sum_{i=1}^{2k}H_{\ge2}(\bsig^i) = EN$
    where $|E - U|\le \eta\sqrt{U}$, whenever $|q_1| \ge \delta$, we have by assumption \eqref{eq:replica-symmetry}, with high probability
    \begin{align*}
\Phi(\bq; (\bsig^i)_{i=1}^{2k}) \le \frac{1}{2}\xi(1)-10\eta.
    \end{align*}
    Integrating over the $(q_i)_{i\le 2k}$ and using Gaussian concentration,
    we deduce that for $|E-U|\le  \eta\sqrt{U}$,
     we have
    \begin{align*}
    \bbP\left\{\left.\int_{\brho:|\langle\brho,\bsig_1\rangle_N| > \delta} e^{H_{\ge2}(\brho)}\mu_0(\de\brho) \le e^{N(\xi(1)/2-9\eta)}\right| \sum_{i=1}^{2k}H_{\ge2}(\bsig^i) = EN\right\} \ge 1-e^{-c(\eta)N}.
    \end{align*}

    Up until now we worked with the Hamiltonian $H_{\ge 2}(\bsig)$, which does not include the term linear in $\bsig$.
    Recall that $H(\bsig) = \<\bu,\bsig\>+H_{\ge 2}(\bsig)$ and $\|\bu\|\le c'\sqrt{N}$, so $
|H(\bsig) - H_{\ge2}(\bsig)| \le |\langle \bu,\bsig\rangle|\le c'N$. Assuming that $c'<\eta$, we have
    \begin{align*}
    \bbP\left\{\left.\int_{\brho:|\langle\brho,\bsig_1\rangle_N| > \delta} e^{H(\brho)}\mu_0(\de\brho) \le e^{N(\xi(1)/2-8\eta)}\right| \sum_{i=1}^{2k}H_{\ge2}(\bsig^i) = EN\right\} \ge 1-e^{-c(\eta)N}.
    \end{align*}
    Hence, under the same conditions
    \begin{align*}
    \bbP\left\{\bsig^1\in T^c\left|\sum_{i=1}^{2k}H_{\ge2}(\bsig^i) = EN\right.\right\} \le
    e^{-c(\eta)N}\, .
    \end{align*}

Define the event
\begin{align*}
\cE(\{\bsig^i\}) :=\left\{\left|\sum_{i=1}^{2k}H_{\ge2}(\bsig^i) -NU\right| \ge N\eta\sqrt{U}\right\}\, .
\end{align*}
    Thus, since $|H(\bsig) - H_{\ge2}(\bsig)|\le c'N$, we can then conclude that
    \begin{align*}
    &\bbE\left\{ \int_{A_{2k}(\delta):\bsig^1\in T^c} e^{\sum_{i=1}^{2k}H(\bsig^i)} \mu_0^{\otimes 2k} (\de \bsig)
    \right\} =\bbE\left\{ \int_{A_{2k}(\delta)} \bfone_{\bsig^1\in T^c}e^{\sum_{i=1}^{2k}H(\bsig^i)} \mu_0^{\otimes 2k} (\de \bsig)
    \right\}\\
    &=\bbE\left\{ \int_{A_{2k}(\delta)} \bbP\Big\{\bsig^1\in T^c\Big|\sum_{i=1}^{2k}H_{\ge2}(\bsig^i) \Big\}e^{\sum_{i=1}^{2k}H(\bsig^i)} \mu_0^{\otimes 2k} (\de \bsig)
    \right\}\\
        &\le e^{-c(\eta)N+2kc'N} \bbE \int_{A_{2k}(\delta)}  e^{\sum_{i=1}^{2k}H_{\ge2}(\bsig^i)} \mu_0^{\otimes 2k} (\de \bsig)  + e^{2kc'N}\bbE \int_{A_{2k}(\delta)}  e^{\sum_{i=1}^{2k}H_{\ge2}(\bsig^i)}
        \bfone_{\cE(\{\bsig^i\})}\mu_0^{\otimes 2k} (\de \bsig)\\
        &\le e^{-c N} \bbE \int_{A_{2k}(\delta)}  e^{\sum_{i=1}^{2k}H(\bsig^i)} \mu_0^{\otimes 2k} (\de \bsig).
    \end{align*}
    Here we assume $c'<c(\eta)/(8k)$ and $c = c(\eta)/4$, and in the last step we used
    \begin{align*}
    \bbE \int_{A_{2k}(\delta)}  e^{\sum_{i=1}^{2k}H_{\ge2}(\bsig^i)}
        \bfone_{\cE(\{\bsig^i\})}\mu_0^{\otimes 2k} (\de \bsig)\le
        2e^{-\eta^2N/2} \int_{A_{2k}(\delta)} e^{NU(\{\bsig^i\})/2} \mu_0^{\otimes 2k}\de\bsig\, ,
    \end{align*}
    and chose $\delta$, $s$ suitably small.
\end{proof}

The next lemma shows that integrals over $S_N^{2k}$ with respect to the product Gibbs measure
are very precisely approximated by an integral over tuples that are very close to orthogonal.
\begin{lem}[Near-orthogonal tuples dominate]\label{lem:OrthogonalTuples}
    For $\delta>0$ sufficiently small and appropriate $c,c_0>0$, if $\|\bu\| \le N^{c_0/2}$, then
    with probability $1-e^{-cN}$ over $\bW^{(2)}$, the following holds:
\begin{enumerate}
\item For quadratic Hamiltonians, the unrestricted partition function of $2k$ replicas is dominated
by its restriction to $A_{2k}(N^{-1/2+c})$:
    \begin{align}\label{eq:quad-error-1}
        Z_{\le 2,2k}\big(A_{2k}(N^{-1/2+c})\big) \ge (1-e^{-N^{c}})\cdot \big(Z_{\le 2}\big)^{2k}.
    \end{align}
\item The contribution of $A_{2k}(\delta)\setminus A_{2k}(N^{-1/2+c}) =\{(\bsig^i)_{i\le 2k}:\max_{i\ne j} |\langle \bsig^i,\bsig^j\rangle_N| \in [N^{-1/2+c}, \delta]\}$ is small:
    \begin{align}
&\bbE_{\ge 3} Z_{2k}\big(A_{2k}(\delta)\setminus A_{2k}(N^{-1/2+c})\big) \le e^{-N^{c}+Nk\xi_{\ge 3}(1)}
\big(Z_{\le 2}\big)^{2k}
.\label{eq:quadratic-tail}
\end{align}
\item Annealing the restricted partition function over $H_{\ge 3}$ is roughly equivalent
to complete annealing:
    \begin{align}\label{eq:quad-error-2}
    \bbE_{\ge3}Z_{2k}\big(A_{2k}(\delta)\big)
     \ge e^{-4k(\|\bu\|+1)\sqrt{N}} \bbE Z_{2k}\big(A_{2k}(\delta)\big) .
    \end{align}
\end{enumerate}
\end{lem}
\begin{proof}
    \noindent{\bf Proof of 1.} By Lemma \ref{lem:overlap-2}, there are constants $c_1,C_1>0$ such that, with probability $1-e^{-cN}$ over $\bW^{(2)}$,
    \begin{align*}
    Z_{\le 2, 2k}\big(S_N^{2k}\setminus A_{2k}(N^{-1/2+c})\big)
        &\le \sum_{i\neq j}\int_{S_N^{2k}}\bfone_{|\<\bsig^i,\bsig^j\>_N|>N^{-1/2+c}}e^{\sum_{i=1}^{2k}H_{\le2}(\bsig^i)} \mu_0^{\otimes 2k}(\de\bsig)\\
        &\le e^{-N^c} \int_{S_N^{2k}}e^{\sum_{i=1}^{2k}H_{\le2}(\bsig^i)} \mu_0^{\otimes 2k}(\de\bsig),
    \end{align*}
    yielding (\ref{eq:quad-error-1}).

    \noindent{\bf Proof of 2.} By a direct calculation, for any set $U\subseteq (S_N)^{2k}$:
    \begin{align*}
        \bbE_{\ge3} Z_{2k}\big(U\big)
        =e^{Nk\xi_{\ge 3}(1)}\int_{U} e^{\sum_{i=1}^{2k}H_{\le2}(\bsig^i)} \exp\lt(N\sum_{i<j\le2k} \xi_{\ge3}(\langle \bsig^i,\bsig^j\rangle_N)\rt)\mu_0^{\otimes 2k} (\de \bsig)\,.
    \end{align*}

    Applying Lemma \ref{lem:overlap-2}, we have for $t>0$ and $\eps_N = N^{-1/2+c}$ that, with probability $1-e^{-cN}$ over $\bW^{(2)}$,
\begin{align}
&\frac{1}{(Z_{\le 2})^{2k}} e^{-Nk\xi_{\ge 3}(1)}\bbE_{\ge 3}Z_{2k}
\Big(A_{2k}(t+\eps_N)\setminus A_{2k}(t)\Big)\le\\
&\frac{1}{(Z_{\le 2})^{2k}}\int_{\max_{i\ne j} |\langle \bsig^i,\bsig^j\rangle_N| \in [t,t+\eps_N]}\exp\lt(\sum_{i=1}^{2k}H_{\le2}(\bsig^{i})+N\sum_{i< j\le 2k}\xi_{\ge3}(\langle\bsig^{i},\bsig^{j}\rangle_N)\rt)\mu_0^{\otimes 2k}(\de\bsig) \le \nonumber \\
&\phantom{AAAAAAAAA AAAAAAAA} \le \exp\lt\{-cN\big(t-\|\bu\|_N^2\big)_+^2+N(2k)^2\xi_{\ge3}(t+\eps_N)\rt\}\, .
\end{align}
Under the assumption $\|\bu\|_N^2\le N^{c_0-1}$, $c_0<c+1/2$, summing over the range $N^{-1/2+c} < |t|\le\delta$, we obtain the following
with probability $1-e^{-cN}$ over $\bW^{(2)}$,
\begin{align*}
&\int_{\max_{i\ne j} |\langle \bsig^i,\bsig^j\rangle_N| \in [N^{-1/2+c}, \delta]}\exp\lt(\sum_{i=1}^{2k}H_{\le2}(\bsig^{i})+N\sum_{i< j\le 2k}\xi_{\ge3}(\langle\bsig^{i},\bsig^{j}\rangle_N)\rt)\mu_0^{\otimes 2k}(\de\bsig)\nonumber\\
&\le \exp(-N^{c})\int_{S_N^{2k}}\exp\lt(\sum_{i=1}^{2k}H_{\le2}(\bsig^{i})\rt)
\mu^{\otimes 2k}_0(\de\bsig).
\end{align*}
This gives (\ref{eq:quadratic-tail}).

 \noindent{\bf Proof of 3.} Note that
\begin{align}
e^{-Nk\xi_{\ge 3}(1)}\bbE_{\ge 3} &Z_{2k}\big( A_{2k}(N^{-1/2+c})\big) =
\nonumber\\
 &=\int_{A_{2k} (N^{-1/2+c})}\exp\lt(\sum_{i=1}^{2k}H_{\le2}(\bsig^{i})+N\sum_{i< j\le 2k}\xi_{\ge3}(\langle\bsig^{i},\bsig^{j}\rangle_N)\rt)\mu_0^{\otimes 2k}(\de\bsig)\nonumber\\
&=(1+O(N^{-1/2+3c}))\cdot Z_{\le 2, 2k}\Big(A_{2k} (N^{-1/2+c})\Big).\label{eq:quad-last}
\end{align}
Therefore, using Eq.~\eqref{eq:quad-error-1}, we get
\begin{align}
\bbE_{\ge 3} &Z_{2k}\big( A_{2k}(N^{-1/2+c})\big) =(1+O(N^{-1/2+3c}))e^{Nk\xi_{\ge 3}(1)}\big(Z_{\le2}\big)^{2k}
\, .\label{eq:Z2kZ2k}
\end{align}

Also,
\begin{align}
    &\bbE Z_{\le 2, 2k}\Big(A_{2k} (N^{-1/2+c})\Big)  \le e^{2k\|\bu\|\sqrt{N}} \exp\big(k\beta_2^2N+(2k)^2\beta_2^2N^{2c}\big). \label{eq:bbE-up}
\end{align}
On the other hand, Lemma \ref{lem:laplace} readily implies that with probability at least $1-e^{-cN}$,
\begin{align}
   (Z_{\le 2})^{2k} &\ge e^{-k\sqrt{N}} \exp(k\beta_2^2N). \label{eq:bbE-low}
\end{align}
Combining Eqs.~\eqref{eq:bbE-up} and \eqref{eq:bbE-low}, we get
\begin{align}
    &\bbE Z_{\le 2, 2k}\big(A_{2k} (N^{-1/2+c})\big)
    \le e^{2k(1+\|\bu\|)\sqrt{N}}   (Z_{\le 2})^{2k} \, .
    \end{align}
Finally, using Eq.~\eqref{eq:Z2kZ2k} together with the last display, we get
 \begin{align}\label{eq:quad-error-2-ALMOST}
    \bbE_{\ge3}Z_{2k}\big(A_{2k}(N^{-1/2+c})\big)
     \ge e^{-3k(\|\bu\|+1)\sqrt{N}}
     \bbE Z_{2k}\big(A_{2k}(N^{-1/2+c})\big) .
    \end{align}
Combining this with Eq.~\eqref{eq:quadratic-tail} yields the claim.

\end{proof}

\begin{lem}\label{lemma:A-T-comparison}
For any $m\ge 2$, there exists a constant $c>0$ such that, for $T=T(\delta)$,
\begin{align}
Z_m\big(A_m(T,\delta)\big) \le Z(T)^m
\le (1+e^{-cN})\cdot Z_m\big(A_m(T,\delta)\big)
+e^{-cN+Nm\xi(1)/2}\, .
\end{align}
\end{lem}
\begin{proof}
The left-hand inequality is obvious since $A_m(T,\delta)\subseteq T^{\otimes m}$. For the right-hand inequality, consider
first the case $m=2$. Then we have
\begin{align}
Z(T)^2
&\le Z_2\big(A_2(T,\delta)\big)
+\int_{T\times T}\bfone_{|\<\bsig^1,\bsig^2\>_N|\ge\delta}
e^{H(\bsig^1)+H(\bsig^2)}\mu_0^{\otimes2}(\de\bsig)\\
&\le Z_2\big(A_2(T,\delta)\big)
+\int_T e^{H(\bsig^1)}
\left[\int_T\bfone_{|\<\bsig^1,\bsig^2\>_N|\ge\delta}
e^{H(\bsig^2)}\mu_0(\de\bsig^2)\right]\mu_0(\de\bsig^1)\\
&\le Z_2\big(A_2(T,\delta)\big)
+\int_T e^{H(\bsig^1)}e^{-c_1(\delta)N+N\xi(1)/2}
\mu_0(\de\bsig^1)\\
&\le Z_2\big(A_2(T,\delta)\big)
+e^{-Nc_1(\delta)+N\xi(1)/2}Z(T)
\label{eq:Z2-beforeLast}\\
&\le Z_2\big(A_2(T,\delta)\big)
+e^{-Nc_1(\delta)+N\xi(1)}
+e^{-Nc_1(\delta)}Z(T)^2\, .
\end{align}
Here, in the last step we used the AM-GM inequality.
Solving this inequality for $Z(T)^2$, we get:
\begin{align}
Z(T)^2\le (1+e^{-cN})Z_2\big(A_2(T,\delta)\big)
+2e^{-c_1(\delta)N+N\xi(1)}\, .
\label{eq:ZTT}
\end{align}
This proves the claim for $m=2$.

Consider now $m\ge2$. Note that
\begin{align*}
&\int_{T(\delta)^m}e^{\sum_{i=1}^mH(\bsig^i)}
\mu_0^{\otimes m}(\de\bsig)
-\int_{A_m(T(\delta),\delta)}e^{\sum_{i=1}^mH(\bsig^i)}
\mu_0^{\otimes m}(\de\bsig)\\
&\le\sum_{i\ne j}
\left(\int_{T(\delta)}e^{H(\bsig)}\mu_0(\de\bsig)\right)^{m-2}
\int_{\bsig^i,\bsig^j\in T(\delta):\,|\<\bsig^i,\bsig^j\>_N|>\delta}
e^{H(\bsig^i)+H(\bsig^j)}
\mu_0(\de\bsig^i)\mu_0(\de\bsig^j).
\end{align*}
Whence
\begin{align*}
Z(T)^m-Z_m\big(A_m(T,\delta)\big)
&\le m^2Z_2\big(T^{\otimes2}\setminus A_2(T,\delta)\big)
\cdot Z(T)^{m-2}\\
&\le m^2\cdot Z(T)^{m-1}\cdot
e^{-Nc_1(\delta)+N\xi(1)/2}\, ,
\end{align*}
where in the last inequality we used Eq.~\eqref{eq:Z2-beforeLast}.
Using again the AM-GM inequality, we get
\begin{align*}
Z(T)^m-Z_m\big(A_m(T,\delta)\big)
\le m^2e^{-Nc_1(\delta)}Z(T)^m
+m^2e^{-Nc_1(\delta)+Nm\xi(1)/2}\, ,
\end{align*}
which yields the claim.
\end{proof}

\section{Proof of Lemma \ref{lem:high-acc-anneal}}
We next prove Lemma \ref{lem:high-acc-anneal}. For the partition-function estimates, fix an integer $k\ge1$ and
$0<L<k$. Choose the overlap-window exponent $c>0$ so small that
$3ck<(k-L)/2$, in addition to the other smallness requirements below.
We expand through degree $k-1$; the Taylor remainder is then
$O(N^{-k/2+3ck})=O(N^{-L/2})$. The magnetization estimate
\eqref{eq:Magn-No-Field} holds for every integer $k\ge1$; when using the
partition-function estimates with this $k$, we may choose any $L\in(0,k)$.
In the proof, we let $c$ denote small absolute constants that can change from line to line. We will first prove the partition
function estimate, Eq.~\eqref{eq:HighAcc-PartitionFun},
and then the magnetization estimate,
Eq.~\eqref{eq:Magn-No-Field}.

\noindent{\bf Estimating the partition function, Eq.~\eqref{eq:HighAcc-PartitionFun}.}  By Eq.~\eqref{eq:quadratic-tail} in Lemma~\ref{lem:OrthogonalTuples}, with probability $1-e^{-cN}$
over $\bW^{(2)}$,
\begin{align}
&\int_{\max_{i\ne j} |\langle \bsig^i,\bsig^j\rangle_N| \in [N^{-1/2+c}, \delta]}\exp\lt(\sum_{i=1}^{2k}H_{\le2}(\bsig^{i})+N\sum_{i< j \le 2k}\xi_{\ge3}(\langle\bsig^{i},\bsig^{j}\rangle_N)\rt)\mu_0^{\otimes 2k}(\de\bsig)\nonumber\\
&\le \exp(-N^{c})\int_{S_N^{2k}}\exp\lt(\sum_{i=1}^{2k}H_{\le2}(\bsig^{i})\rt)
\mu^{\otimes 2k}_0(\de\bsig).
\end{align}
On $A_{2k}(N^{-1/2+c}) = \{|\langle\bsig^{i},\bsig^{j}\rangle_N|\le N^{-1/2+c}\;\; \forall
i\ne j\}$, we can expand
\begin{align*}
\exp\left(N\sum_{i < j}\xi_{\ge3}(\langle\bsig^{i},\bsig^{j}\rangle_N)\right)
=\sum_{\ell=0}^{k-1}\frac{1}{\ell!}\Big(N\sum_{i<j}\xi_{\ge3}(\langle\bsig^{i},\bsig^{j}\rangle_N)
\Big)^{\ell}+
O(N^{-k/2+3ck}).
\end{align*}
Thus, for $T=T(\delta)$, the following holds with probability at least $1-e^{-cN}$ over $\bW_2$,
\begin{align}
& \bbE_{\ge3}\Big\{\big(Z(T)-\bbE_{\ge3}Z(T)\big)^{2k}\Big\}\\
&\stackrel{(a)}{\le} \sum_{r\le2k}\binom{2k}{2k-r}(-1)^{r}\big(\bbE_{\ge3}Z(T)\big)^{2k-r}\cdot\bbE_{\ge3}Z\big(A_r(T,\delta)\big)\nonumber\\
&\phantom{AAA}
+e^{-Nc}\sum_{r\le2k}\binom{2k}{2k-r}\big(\bbE_{\ge3}Z(T)\big)^{2k-r}\cdot\Big(\bbE_{\ge3}Z\big(A_r(T,\delta)\big) + e^{Nr\xi(1)/2}\Big) \nonumber\\
&\stackrel{(b)}{\le} \sum_{r\le2k}\binom{2k}{2k-r}(-1)^{r}\big(\bbE_{\ge3}Z(T)\big)^{2k-r}\cdot\bbE_{\ge3}Z\big(A_r(T,\delta)\big)\\
&\phantom{AAA}
+e^{-Nc}\max_{r\le2k}e^{N(2k-r)\xi(1)/2}\cdot\Big(\bbE_{\ge3}Z\big(A_r(T,\delta)\big) + e^{Nr\xi(1)/2}\Big)
\, ,\label{eq:Zdev2}
\end{align}
where $(a)$ follows from Lemma \ref{lem:OrthogonalTuples} and $(b)$ holds because $\bbE_{\ge3}Z(T)
\le e^{c'N}\bbE Z$ with the claimed probability  by Markov's inequality.

We define the error terms
\begin{align}
    \Err_1 & := e^{-cN+Nk\xi(1)} +
    e^{-cN}\max_{1\le r\le 2k}\Big(\bbE_{\ge3} Z(A_{r}(T,\delta))\Big)^{2k/r}+
    \bbE_{\ge3}Z(A_{2k}(\delta)\cap\{\bsig^1 \in T^c\}\big),\\
       \Err_2 & := N^{-L/2} e^{Nk\xi_{\ge3}(1)}Z_{\le 2, 2k}(A_{2k}(\delta))\, ,
\end{align}
so that the bound \eqref{eq:Zdev2} implies
\begin{align}
& \bbE_{\ge3}\Big\{\big(Z(T)-\bbE_{\ge3}Z(T)\big)^{2k}\Big\}\le \sum_{r\le2k}\binom{2k}{2k-r}(-1)^{r}\big(\bbE_{\ge3}Z(T)\big)^{2k-r}\cdot\bbE_{\ge3}Z\big(A_r(T,\delta)\big)+O_k(\Err_1)\, .
\label{eq:Zdev3}
\end{align}
Next note that
\begin{align*}
&\big(\bbE_{\ge3}Z(T)\big)^{2k-r}\cdot\bbE_{\ge3}Z\big(A_{r}(T,\delta)\big)
\\
&=e^{Nk\xi_{\ge 3}(1)}\lt(\int_{T(\delta)} e^{H_{\le2}(\bsig)}\mu_0(\de\bsig)\rt)^{2k-r}\cdot\\
 & \qquad\qquad
 \cdot\int_{A_{r}(\delta)}
 e^{\sum_{i=1}^{r}H_{\le2}(\bsig^{i})}\exp\lt(N\sum_{i< j}\xi_{\ge3}(\langle\bsig^{i},\bsig^{j}\rangle_N)\rt)\mu^{\otimes r}_0(\de\bsig)  + O_k(\Err_1) \\
 &= e^{Nk\xi_{\ge 3}(1)}\int_{A_{2k}(\delta)}  \exp\left(\sum_{i'=1}^{2k-r} H_{\le 2}((\bsig')^{i'}) + \sum_{i=1}^{r}H_{\le 2}(\bsig^i)\right) \cdot \\
 &\qquad \qquad \cdot \left\{\sum_{\ell=0}^{k-1} \frac{1}{\ell!} \left(N\sum_{i< j\le r}\xi_{\ge3}(\langle \bsig^i,\bsig^j\rangle_N)\right)^{\ell}\right\} \mu^{\otimes r}_0(\de\bsig)\mu^{\otimes (2k-r)}_0(\de\bsig') + O_k(\Err_1+\Err_2),
\end{align*}
where the last inequality holds with probability $1-e^{-cN}$ over $\bW^{(2)}$
 by Eq.~\eqref{eq:quadratic-tail}.

Substituting in Eq.~\eqref{eq:Zdev3}, we get
\begin{align}
&\bbE_{\ge3}\Big\{\big(Z(T)-\bbE_{\ge3}Z(T)\big)^{2k}\Big\}\nonumber\\
 &\le e^{Nk\xi_{\ge3}(1)}\int_{A_{2k}(\delta)}e^{\sum_{i=1}^{2k}H_{\le2}(\bsig^{i})}
 \sum_{\ell\le k-1}\frac{1}{\ell!}\sum_{r\le2k}(-1)^{r}\sum_{S\subseteq [2k]:|S|=r}
 \Big(N\sum_{i< j\in S}\xi_{\ge3}(\langle\bsig^{i},\bsig^{j}\rangle_N)\Big)^{\ell}\mu_0^{\otimes 2k}(\de\bsig)   \nonumber\\
 & \qquad \qquad \qquad + O_k(\Err_1+\Err_2). \label{eq:Taylor}
\end{align}

We can expand  the $\ell$-th power in \eqref{eq:Taylor}, thus getting a sum indexed
by sets of pairs  $S= \{(i_{t},j_{t}):t\le\ell\}\subseteq \binom{[2k]}{2}$. Denoting by
$n(S)$ the number of distinct elements of $[2k]$ appearing in $S$, the coefficient
of such a term is
\[
\sum_{n(S)\le r\le2k}(-1)^{r}\binom{2k-n(S)}{r-n(S)}=0
\]
for $|\{i_{t},j_{t}:t\le\ell\}|<2k$. Hence, we have
\beq
\label{eq:E-ge3-variance-bound}
\bbE_{\ge3}\Big\{\big(Z(T)-\bbE_{\ge3}Z(T)\big)^{2k}\Big\}=O_k(\Err_1+\Err_2)\, .
\eeq

We now estimate the error terms.

\noindent{\em Error term $\Err_2$.}
Using Lemma \ref{lem:truncate}, we have
\begin{align*}
    \big(\bbE_{\ge3}Z(T)\big)^{2k} &= \Big(\bbE_{\ge3}Z - \bbE_{\ge3}Z(T^c)\Big)^{2k} \nonumber\\
    &\ge (1-e^{-c N/8}) \big(\bbE_{\ge3}Z\big)^{2k}  \\
    &\ge c\, e^{Nk\xi_{\ge3}(1)}\big(Z_{\le 2}\big)^{2k}\\
    & \ge c\, e^{Nk\xi_{\ge3}(1)} Z_{\le 2}\big(A_{2k}(\delta)\big).
\end{align*}
From this estimate, we obtain with probability at least $1-\exp(-cN/8)$ over $\bW^{(2)}$ that
\begin{align}
    \Err_2 &\le C\cdot N^{-L/2}\cdot \big(\bbE_{\ge 3 }Z(T)\big)^{2k} .\label{eq:Err2Final}
\end{align}

\noindent{\em Error term $\Err_1$.} Using Lemma \ref{lem:key-error}
and Markov's inequality, with probability $1-\exp(-c N/2)$ over $\bW^{(2)}$,
\begin{align}
    &\bbE_{\ge3}  Z\big(A_{2k}(\delta)\cap\{\bsig^1 \in T^c\}\big) \le e^{-cN/2} \bbE
    Z(A_{2k}(\delta)).
\end{align}
Further using  Eq.~\eqref{eq:quad-error-2} in Lemma \ref{lem:OrthogonalTuples}, and using the assumption on
$\|\bu\|_2$,
 with probability $1-\exp(-c N/4)$ over $\bW^{(2)}$,
we have
\begin{align}
    &\bbE_{\ge3}  Z\big(A_{2k}(\delta)\cap\{\bsig^1 \in T^c\}\big) \le e^{-cN/2} \bbE_{\ge3}
    Z(A_{2k}(\delta)). \label{eq:key-error-2-annulus}
\end{align}
Hence, with probability at least $1-\exp(-c N/8)$ over $\bW^{(2)}$,
\begin{align}\label{eq:Err1-almost}
  \Err_1 & \le e^{-cN+Nk\xi(1)} +
    e^{-cN}\max_{1\le r\le 2k}\Big(\bbE_{\ge3} Z(A_{r}(\delta))\Big)^{2k/r}
\end{align}

Further, with probability at least $1-\exp(-cN/8)$ over $\bW^{(2)}$,
\begin{align}
    \bbE_{\ge3} Z\big(A_{r}(\delta)\big)  &= \bbE_{\ge3} Z\big(A_{r}(N^{-1/2+c})\big)+
    \bbE_{\ge3} Z\big(A_r(\delta)\setminus A_{r}(N^{-1/2+c})\big)\nonumber\\
    & \le 2\, e^{Nr\xi_{\ge 3}(1)/2}  Z_{\le 2}\big(A_{r}(N^{-1/2+c})\big)+ e^{-N^c+Nr\xi_{\ge 3}(1)/2} (Z_{\le 2})^{r}
    \, ,\label{eq:secErrorTerm}
\end{align}
where in the last line we used Eq.~\eqref{eq:quadratic-tail}, and the fact that
\begin{align*}
\bbE_{\ge3} Z\big(A_{r}(N^{-1/2+c})\big)
&= \int_{A_{r}(N^{-1/2+c})} e^{\sum_{i=1}^{r}H_{\le 2}(\bsig^i)}
    \exp\left(\frac{N}{2}\sum_{i,j\le r}\xi_{\ge 3}(\<\bsig^i,\bsig^j\>_N )\right)\mu_0^{\otimes r}(\de \bsig)\\
& \le \big(1+O(N^{-1/2+3c})\big)\, e^{Nk\xi_{\ge 3}(1)}  Z_{\le 2}\big(A_{r}(N^{-1/2+c})\big)   \, .
\end{align*}
Using Eq.~\eqref{eq:laplace-1} in Eq.~\eqref{eq:secErrorTerm}, we get
\begin{align}
  \bbE_{\ge3} Z\big(A_{r}(\delta)\big) \le N^C e^{Nr\xi(1)/2}\, ,
\end{align}
whence Eq.~\eqref{eq:Err1-almost} simplifies to
\begin{align}
\Err_1 &\le e^{-cN+Nr\xi(1)}\, .\label{eq:Err1-almost2}
\end{align}
On the other hand, by Lemma \ref{lem:truncate} and Markov's inequality, with probability $1-\exp(-c N/4)$ over $\bW^{(2)}$,
\begin{align*}
    \bbE_{\ge3}\int_{T(\delta)^c} e^{H(\bsig)} \mu_0(\de\bsig) &\le e^{-cN/4} e^{N\xi(1)/2}.
\end{align*}
Using Lemma \ref{lem:laplace}, we obtain that, with probability at least $1-\exp(-cN/8)$ over $\bW^{(2)}$,
\begin{align*}
    \bbE_{\ge3} Z &= \int e^{N\xi_{\ge3}(1)/2} e^{H_{\le2}(\bsig)}\mu_0(\de\bsig) \\
    &\ge e^{N\xi(1)/2-cN/10}\, ,
\end{align*}
whence Eq. \eqref{eq:Err1-almost2} yields
\begin{align}
\Err_1 &\le e^{-cN/16}\big(\bbE_{\ge3} Z(T)\big)^{2k}.\label{eq:Err1Final}
\end{align}

We also note here the estimate
\begin{align}
    \bbE_{\ge3} Z(T) = \bbE_{\ge3} Z - \bbE_{\ge3} Z(T^c) \ge (1-e^{-cN/10})\bbE_{\ge3} Z,
\end{align}
which holds with probability at least $1-\exp(-cN/8)$ over $\bW^{(2)}$, as claimed in Eq.~(\ref{eq:Z3T}).

Combining the error estimates \eqref{eq:Err1Final} and \eqref{eq:Err2Final}
in the moment bound \eqref{eq:E-ge3-variance-bound},
we get, with probability at least
$1-\exp(-Nc)$ with respect to $\bW_2$,
\begin{align}
\bbE_{\ge3}\Big\{\big(Z(T)-\bbE_{\ge3}Z(T)\big)^{2k}\Big\}
\le C\, N^{-L/2}\, \big(\bbE_{\ge3}Z(T)\big)^{2k}\, .
\label{eq:ZT_MomentBound-}
\end{align}
Adjusting $c$, we have
\[
\bbP\left(\left|Z-\bbE_{\ge3}Z(T(\delta))\right|>\eps\, \bbE_{\ge3}Z(T(\delta))\right)\le
\eps^{-2L}N^{-L/2}+ e^{-cN}.
\]

\noindent{\bf Estimating the magnetization, Eq.~\eqref{eq:Magn-No-Field}.}
We next apply the same argument to the magnetization.
First, we note that
\begin{align}
  \bbE_{\ge3}\left\{\left(\int_{T(\delta)}\sigma_{1}e^{H(\bsig)}\mu_0(\de\bsig)\right)^{2k}\right\}
 & =\bbE_{\ge3}\int_{T(\delta)^{2k}}\prod_{i=1}^{2k}\sigma_{1}^{i}\, e^{\sum_{i=1}^{2k}H(\bsig^{i})} \, \mu_0^{\otimes 2k}(\de\bsig)\\
 & =\bbE_{\ge3}\int_{A_{2k}(\delta)}\prod_{i=1}^{2k}\sigma_{1}^{i}\,
 e^{\sum_{i=1}^{2k}H(\bsig^{i})}\mu^{\otimes 2k}(\de\bsig) + \Err_{3}\, ,\label{eq:MagnFirstDecomp}
 \end{align}
where
\begin{equation}\label{eq:Err3}
    \Err_{3} := \bbE_{\ge3}\int_{T(\delta)^{2k}}\prod_{i=1}^{2k}\sigma_{1}^{i}\,
 e^{\sum_{i=1}^{2k}H(\bsig^{i})}\mu^{\otimes 2k}(\de\bsig)
 - \bbE_{\ge3}\int_{A_{2k}(\delta)}\prod_{i=1}^{2k}\sigma_{1}^{i}\,
 e^{\sum_{i=1}^{2k}H(\bsig^{i})}\mu^{\otimes 2k}(\de\bsig) .
\end{equation}
We have
\begin{align}
&\lt|\int_{T(\delta)^{2k}} \prod_{i=1}^{2k}\sigma_1^ie^{\sum_{i=1}^{2k}H(\bsig^{i})}\mu_0^{\otimes 2k}(\de\bsig) - \int_{A_{2k}(T(\delta),\delta)}\prod_{i=1}^{2k}\sigma_1^ie^{\sum_{i=1}^{2k}H(\bsig^{i})}\mu_0^{\otimes 2k}(\de\bsig)\rt| \nonumber\\
 &= \lt|\int_{T(\delta)^{2k}: \max_{i\ne j}|\<\bsig^i,\bsig^j\>_N|>\delta} \prod_{i=1}^{2k}\sigma_1^ie^{\sum_{i=1}^{2k}H(\bsig^{i})}\mu_0^{\otimes 2k}(\de\bsig)\rt| \nonumber\\
 & \stackrel{(a)}{\le} N^k (2k)^2 e^{-c N + N\xi(1)/2}
 Z\big(T(\delta)\big)^{2k-1} \nonumber \\
 & \stackrel{(b)}{\le}  e^{-cN/(2k)+Nk\xi(1)} + e^{-cN/(2k)}
 Z\big(T(\delta)\big)^{2k}\nonumber \\
 & \stackrel{(c)}{\le}  e^{-cN/(2k)+Nk\xi(1)} + e^{-cN/(2k)}
 \lt(\bbE_{\ge 3} Z\big(T(\delta)\big)\rt)^{2k}\, ,\label{eq:MM-up-magnetization}
\end{align}

Here in $(a)$ we used the definition of $T(\delta)$ and in
$(b)$ the AM-GM inequality.
aking $\bbE_{\ge3}$ in \eqref{eq:MM-up-magnetization},
the centered moment bound \eqref{eq:ZT_MomentBound} gives
$\bbE_{\ge3}Z(T)^{2k}\le C_k(\bbE_{\ge3}Z(T))^{2k}$.
The contribution from $A_{2k}(\delta)\setminus A_{2k}(T,\delta)$
is controlled by Lemma~\ref{lem:key-error} and Markov's inequality,
as in \eqref{eq:key-error-2-annulus}. Using Lemma~\ref{lem:laplace}
to absorb the deterministic exponential term yields
\eqref{eq:Err3-bound}, with probability $1-e^{-cN}$ over the
quadratic disorder:

 \begin{align}
 |\Err_{3}| \le e^{-cN}\lt(\bbE_{\ge3}Z\big(T(\delta)\big)\rt)^{2k}. \label{eq:Err3-bound}
\end{align}

Turning to the main term in Eq.~\eqref{eq:MagnFirstDecomp},
\begin{align*}
    &\bbE_{\ge3}\int_{A_{2k}(\delta)}\prod_{i=1}^{2k}\sigma_{1}^{i}\,
 e^{\sum_{i=1}^{2k}H(\bsig^{i})}\mu^{\otimes 2k}(\de\bsig)\\
    &=e^{Nk\xi_{\ge3}(1)}\int_{A_{2k}(\delta)} \prod_{i=1}^{2k} \sigma_1^i \exp\left\{\sum_{i=1}^{2k}H_{\le2}(\bsig^{i})+\frac{N}{2}\sum_{i\ne j}\xi_{\ge3}(\<\bsig^{i},\bsig^{j}\>_N)
\right\} \mu_0^{\otimes 2k}(\de \bsig)\, .
\end{align*}
By Eqs.~\eqref{eq:quad-error-1} and \eqref{eq:quadratic-tail}
in Lemma \ref{lem:OrthogonalTuples}, we can bound
\begin{align*}
    &\bbE_{\ge3}\int_{A_{2k}(\delta)}\prod_{i=1}^{2k}\sigma_{1}^{i}\,
 e^{\sum_{i=1}^{2k}H(\bsig^{i})}\mu^{\otimes 2k}(\de\bsig)\\
    &=e^{Nk\xi_{\ge3}(1)}\int_{A_{2k}(N^{-1/2+c})} \prod_{i=1}^{2k} \sigma_1^i \exp\left\{\sum_{i=1}^{2k}H_{\le2}(\bsig^{i}) + \frac{N}{2}\sum_{i\ne j}\xi_{\ge3}(\< \bsig^i,\bsig^j\>_N)
\right\} \mu_0^{\otimes 2k}(\de \bsig) \\
&\qquad \qquad + O\lt( N^{k} e^{-N^c+Nk\xi_{\ge3}(1)}
Z_{\le 2,2k}\big(A_{2k}(\delta)\big)\rt).
\end{align*}
To bound the first term, we use Lemma \ref{lem:overlap-2} to obtain
\begin{align*}
    &\int_{A_{2k}(\delta)} \prod_{i=1}^{2k} \sigma_1^i \exp\left\{\sum_{i=1}^{2k}H_{\le2}(\bsig^{i})
\right\} \mu_0^{\otimes 2k}(\de \bsig) \\
&= \int_{S_N^{2k}} \prod_{i=1}^{2k} \sigma_1^i \exp\left\{\sum_{i=1}^{2k}H_{\le2}(\bsig^{i})
\right\} \mu_0^{\otimes 2k}(\de \bsig) + O_k\lt(N^k
e^{-c\delta^2N}\big(Z_{\le 2}\big)^{2k}\rt)\\
&= \lt(\int_{S_N} \sigma_1 \exp\left\{H_{\le2}(\bsig)
\right\} \mu_0^{}(\de \bsig)\rt)^{2k} + O_k\lt(N^k
e^{-c\delta^2N}\big(Z_{\le 2})^{2k}\rt).
\end{align*}
By Lemma \ref{lem:laplace}, we then obtain
\begin{align}
    &\int_{A_{2k}(\delta)} \prod_{i=1}^{2k} \sigma_1^i \exp\left\{\sum_{i=1}^{2k}H_{\le2}(\bsig^{i})
\right\} \mu_0^{\otimes 2k}(\de \bsig)\le C_k \lt({|u_1|}^{2k}+N^k e^{-c\delta^2N}\rt)\lt(Z_{\le 2}\rt)^{2k}. \label{eq:A2k-magnetization}
\end{align}

On the other hand, by taking the Taylor expansion of $\exp\lt\{\frac{N}{2}\sum_{i\ne j}\xi_{\ge3}(\<\bsig^i,\bsig^j\>_N)\rt\}$ up to terms of order $L=Ck$ for $C>2$, we obtain that, for $\xi_{\ge3, \le \ell}(s) = \sum_{3\le p\le \ell} \beta_p^2 s^p$,
\begin{align}
    &\frac{1}{(Z_{\le 2})^{2k}}\int_{A_{2k}(N^{-1/2+c})} \prod_{i=1}^{2k} \sigma_1^i \lt(\exp\lt(N\sum_{i< j}\xi_{\ge3}(\<\bsig^i,\bsig^j\>_N)\rt)-1\rt)
    e^{\sum_{i=1}^{2k}H_{\le2}(\bsig^{i})} \mu_0^{\otimes 2k}(\de \bsig)\nonumber \\
&= O(N^{-k}) + \frac{1}{(Z_{\le 2})^{2k}}\sum_{\ell\le L} \frac{N^{\ell}}{\ell!}\int_{A_{2k}(N^{-1/2+c})} \prod_{i=1}^{2k} \sigma_1^i \lt(\sum_{i< j}\xi_{\ge3}(\<\bsig^i,\bsig^j\>_N)\rt)^{\ell}
e^{\sum_{i=1}^{2k}H_{\le2}(\bsig^{i})}
\mu_0^{\otimes 2k}(\de \bsig)\nonumber\\
&= O(N^{-k}) + \frac{1}{(Z_{\le 2})^{2k}}\sum_{\ell\le L} \frac{N^{\ell}}{\ell!}\int_{A_{2k}(N^{-1/2+c})} \prod_{i=1}^{2k} \sigma_1^i \lt(\sum_{i< j}\xi_{\ge3,\le 4k}(\<\bsig^i,\bsig^j\>_N)\rt)^{\ell}
e^{\sum_{i=1}^{2k}H_{\le2}(\bsig^{i})}
\mu_0^{\otimes 2k}(\de \bsig)\nonumber\\
&\stackrel{(a)}{=} O(N^{-k}+e^{-N^c}) + \frac{1}{(Z_{\le 2})^{2k}}\sum_{\ell\le L} \frac{N^{\ell}}{\ell!}\int_{S_N^{2k}} \prod_{i=1}^{2k} \sigma_1^i \lt(\sum_{i< j}\xi_{\ge3,\le 2k}(\<\bsig^i,\bsig^j\>_N)\rt)^{\ell}e^{\sum_{i=1}^{2k}H_{\le2}(\bsig^{i})} \mu_0^{\otimes 2k}(\de \bsig)\nonumber\\
&\stackrel{(b)}{=} O(N^{-k}+e^{-N^c}) + (1+\|\bu\|)^{O(k^2)} O_{k}\lt(\sum_{\ell \le k}|u_1|^{2k-2\ell}N^{-\ell} + \sum_{k<\ell \le L} N^{-\ell/2-k/2}\rt) \\
&= (1+\|\bu\|)^{O(k^2)}O(N^{-k}+e^{-N^c}+|u_1|^{2k})\, ,
\label{eq:LastEst-Magn}
\end{align}
where in $(a)$ we used again
Lemma \ref{lem:OrthogonalTuples} and in $(b)$
Lemma \ref{lem:laplace-moments}.

We thus have from Eqs.~\eqref{eq:Err3-bound},
\eqref{eq:A2k-magnetization} and \eqref{eq:LastEst-Magn},
\begin{align*}
&\bbE_{\ge3}\left\{\Big(\int_{T(\delta)}\sigma_{1} e^{H(\bsig)}\mu_0(\de\bsig)\Big)^{2k}\right\} \\
&\le C_k(1+\|\bu\|)^{Ck^2}\lt({|u_1|}^{2k}+N^{-k}+N^ke^{-N^c}+e^{-cN}\rt)\lt(\bbE_{\ge3}
Z\big(T(\delta)\big)\rt)^{2k}.
\end{align*}
The desired claim \eqref{eq:Magn-No-Field} follows from Markov's inequality upon adjusting the constant $c$.

\section{Magnetization in a weak field: Proof of Theorem
\ref{lem:correction-band}}

In the remainder of this section, we denote by $\mu$ the Gibbs measure associated to $H(\bsig)$, i.e.
$$\mu(\de\bsig) \propto \exp(H(\bsig))\, \mu_0(\de\bsig).$$

In the following we estimate the components of
$\<\bsig\> = (\<\sigma_1\>,\dots,\<\sigma_N\>)$
in the basis of eigenvectors of the quadratic part of the Hamiltonian $\bW^{(2)}$.
For simplicity of notation, we consider the
component $\<\sigma_1\>$, but we emphasize that
this does not necessarily correspond to the largest
(or smallest) eigenvalue of $\bW^{(2)}$.
Defining $\bsig_{-1}=(\sigma_{2},\dots,\sigma_{N})$, we have
\begin{align}
\int\sigma_{1} e^{H(\bsig)}\mu_0(\de\bsig)
= \frac{1}{Z} \int\sigma_{1}e^{\sigma_{1}u_{1}+\Lambda_{1}\sigma_{1}^{2}}\hat{E}(\sigma_{1})
\de\sigma_1\, .
\end{align}
Here, we define
\[
\hat{E}(\sigma_{1})= C_N
(1-\sigma_{1}^2/N)^{(N-3)/2}\int \exp\Big(\frac{\sigma_1}{N}\sum_{i,j=2}^N\tg_{1ij}\sigma_{i}\sigma_{j}\Big)
e^{H_{\sigma_1}(\bsig_{-1})}
\mu_{0,\sqrt{N-\sigma_{1}^2}}(\de\bsig_{-1})\, .
\]
The measure $\mu_{0,\rho}$ is uniform over the sphere of radius $\rho$. We also define
\[
C_N:=\frac{\Gamma(N-1)}{\Gamma((N-1)/2)^22^{N-2}\sqrt{N}} =
\frac{1}{\sqrt{2\pi}}+O(N^{-1})\, .
\]
Finally, we define
\[
H_{\sigma_1}(\bsig_{-1}) := \sum_{i=2}^{N} (\sigma_iu_i + \Lambda_i\sigma_i^2) + N^{-1}\sum_{i,j,k>1} g^{(3)}_{ijk}\sigma_i\sigma_j\sigma_k + \sum_{p\ge4} H_p(\bsig) + R_1(\bsig).\]

Here
\[
\tg_{1ij}=3N W^{(3)}_{1ij}\quad(i,j>1),\qquad
R_1(\bsig)=\frac{\sigma_1^2}{N}\sum_{i>1}d_i\sigma_i+\frac{g_{111}}{N}\sigma_1^3,
\qquad d_i=3N W^{(3)}_{11i}.
\]
The variables $(d_i)_{i>1}$ are independent $\cN(0,3\beta_3^2)$, and are independent of $(\tg_{1ij})_{i,j>1}$ and $g_{111}$.

In particular $\tg_{1ij} = \tg_{1ji}$ and
\begin{align}
(\tg_{1ij})_{1<i<j} \sim_{iid}\cN(0,3\beta_3^2/2)\,,\;\;\;\;
(\tg_{1ii})_{1<i} \sim_{iid}\cN(0,3\beta_3^2)\, .
\end{align}

We set $\hat{E}(\sigma_{1})=0$ for $|\sigma_1|>\sqrt{N}$.

By Lemma \ref{lem:laplace}  and Lemma \ref{lem:high-acc-anneal}, choosing $L$ sufficiently large for the prescribed $C$, then an
integer $k>L$, and then $c_0$ sufficiently small, with probability $1-e^{-cN} - N^{-C}$,
\begin{align*}
Z &=(1+O(N^{-c}))\sqrt{\frac{2}{G''(\gamma_*)}}
\cdot\exp\Big\{N\Big[\frac{1}{2}\xi_{\ge3}(1) -\frac{1}{2}\log (2e)+G(\gamma_*)\Big]\Big\}\, ,
\end{align*}
where $G(\gamma)$ and $\gamma_{*}$ were defined in Eqs.~\eqref{eq:GDef} and \eqref{eq:GammaDef}.

In estimating $\<\sigma_{1}\>$, we first anneal over $\bg_{\ge4}$
and $\bg_{3-}:=(g_{ijk}: i,j,k>1)$. We have
\begin{align*}
E(\sigma_1):=\bbE_{\bg_{3-},\bg_{\ge4}}[\hat{E}(\sigma_{1})] &= C_N\Big(1-\frac{\sigma_{1}^2}{N}\Big)^{(N-3)/2}
\int\exp\left(\frac{\sigma_{1}}{N}\sum_{i,j=2}^N\tg_{1ij}\sigma_{i}\sigma_{j}\right)\\
&\qquad \qquad
\exp\Big\{H_{\le2}(\bsig_{-1})+N\xi_{\ge4}(1)/2+N\beta_{3}^{2}(1-\sigma_{1}^{2}/N)^{3}/2+R_1(\bsig)\Big\}
\mu_{0,\sqrt{N-\sigma_{1}^2}}(\de\bsig_{-1}).
\end{align*}
The next lemma shows that this expectation is an accurate approximation of
$\hat{E}(\sigma_{1})$.
\begin{lem}
\label{lem:anneal}We have for an appropriate $c \in (0,1/8)$ that, with probability $1-N^{-c}$,
\begin{align*}
 & \int\sigma_{1} e^{u_{1}\sigma_{1}+\Lambda_{1}\sigma_{1}^{2}}\hat{E}(\sigma_{1})\, \de\sigma_{1}=\\
  & = \int\sigma_{1} e^{u_{1}\sigma_{1}+\Lambda_{1}\sigma_{1}^{2}}E(\sigma_{1})\de\sigma_{1}
 + O\lt(N^{-1/2+c}(|u_1|+N^{-1/2}) \int e^{u_{1}\sigma_{1}+\Lambda_{1}\sigma_{1}^{2}}E(\sigma_1)\de\sigma_{1}\rt).
\end{align*}

On the same event, writing $\overline Z=\int e^{u_1s+\Lambda_1s^2}E(s)\,\de s$, we also have
\begin{align}\label{eq:partial-anneal-denominator}
Z=(1+O(N^{-c}))\overline Z.
\end{align}

\end{lem}

Before proving Lemma \ref{lem:anneal}, we use it to prove Theorem \ref{lem:spin}.
\begin{proof}[Proof of Theorem \ref{lem:spin}]
For $U(\sigma_1):=N\xi_{\ge4}(1)/2+N\beta_{3}^{2}(1-\sigma_{1}^{2}/N)^{3}/2$, we
have
\[
E(\sigma_{1})=C_N
\Big(1-\frac{\sigma_{1}^2}{N}\Big)^{(N-3)/2} \int
\exp\left(\frac{\sigma_{1}}{N}\sum_{i,j=2}^N\tg_{1ij}\sigma_{i}\sigma_{j}\right)\, e^{H_{\le2}(\bsig_{-1})+U(\sigma_1)}
\, \mu_{0,\sqrt{N-\sigma_{1}^2}}(\de\bsig_{-1})\, .
\]
Again by the calculation in Lemma \ref{lem:laplace}, for $\bV =\bV(\sigma_1) := \bLambda_{-1}+\bDel$, where $\bLambda_{-1}$ is the
diagonal matrix with entries corresponding to the
spectrum of $\bW^{(2)}$, with $\Lambda_1$ replaced by $0$, and
$\bDel:=\sigma_1N^{-1}\tilde{\bG}$ with $\tilde{G}_{ij} = \tg_{1ij}$, and letting $\bu_{\sigma_1}=\bu_{-1}+\sigma_1^2\boldsymbol d/N$,
we have

\begin{align}\label{eq:Esigma1}
E(\sigma_1)
&=(1+O(N^{-1+Cc_0}))(2e)^{-N/2}
\sqrt{\frac{2}{\pi G_{\sigma_1}''(\gamma_*(\sigma_1))}}
\exp\big(U(\sigma_1)+g_{111}\sigma_1^3/N+NG_{\sigma_1}(\gamma_*(\sigma_1))\big),
\end{align}

where we defined
\begin{align}
G_{\sigma_1}(\gamma) &:= (1-\sigma_1^{2}/N)\gamma-\frac{1}{2N}\log\det(\gamma \bI_{N-1}-\bV)+\frac{1}{4N}\<\bu_{\sigma_1},
(\gamma \bI_{N-1}-\bV)^{-1}\bu_{\sigma_1}\>\, ,\\
\gamma_*(\sigma_1) &= \arg\min_{\gamma>\lambda_{\max}(\bV)} G_{\sigma_1}(\gamma)\, .
\end{align}

Here and below $|\sigma_1|\le a_N:=N^{2c_0}\log N$.

 By Lemma \ref{lem:anneal}, we have
\begin{align}
    \label{eq:ApplyAnneal}
\int\sigma_{1}\,\mu(\de\bsig)=
(1+O(N^{-c}))\frac{\int \sigma_1e^{u_{1}\sigma_{1}+\Lambda_{1}\sigma_{1}^{2}}
E(\sigma_1)\de\sigma_{1}}{\int e^{u_{1}\sigma_{1}+\Lambda_{1}\sigma_{1}^{2}}
E(\sigma_1)\de\sigma_{1}}+O\lt(N^{-1/2+c}(|u_1|+N^{-1/2})\rt)\, .
\end{align}
We next estimate these integrals by
approximating their integrands for small
$\sigma_1$. Note that by Lemma \ref{lem:overlap-2}
and Lemma \ref{lem:truncate}, we can restrict these
integrals to $|\sigma_1|\le C\log N$, making a negligible error.

It is easy to see that, for $\sigma_1=0$, we recover $G_{\sigma_1}(\gamma)=G_0(\gamma)$, where $G_0(\gamma)$ is the same function
defined in Eq.~\eqref{eq:GDef}, with $N$ replaced by $N-1$.
To leading order, we can expand
\begin{align*}
&G_{\sigma_{1}}(\gamma) =\\ & = (1-\sigma_{1}^{2}/N)\gamma-\frac{1}{2N}\log\det(\gamma \bI-\bV)+\frac{1}{4N}\<\bu_{\sigma_1},(\gamma \bI-\bV)^{-1}\bu_{\sigma_1}\> \\
 & =(1-\sigma_{1}^{2}/N)\gamma-\frac{1}{2N}\log\det(\gamma \bI-\bV)+\frac{1}{4N}\<\bu_{\sigma_1}(\bI+(\gamma \bI-\bLambda_{-1})^{-1}\bDel+\bE_N)(\gamma \bI-\bLambda_{-1})^{-1}\bu_{\sigma_1}\>.
\end{align*}
Here, $\|\bE_N\|_{\op} = O(N^{-1+Cc_0})$ with probability
$1-\exp(-cN)$ over $\bW^{(3)}$. Writing $\bR_\gamma=(\gamma\bI-\bLambda_{-1})^{-1}$, independence and Gaussian concentration give
\[
|\boldsymbol d^T(\gamma\bI-\bV)^{-1}\bu_{-1}|\le N^{c_0}(1+\|\bu\|),\qquad
|\boldsymbol d^T(\gamma\bI-\bV)^{-1}\boldsymbol d|\le CN.
\]
Thus replacing $\bu_{\sigma_1}$ by $\bu_{-1}$ costs $O(N^{-2+Cc_0})$ in $G_{\sigma_1}$. Therefore,

\begin{align}\label{eq:GdiffFirst}
G_{\sigma_1}(\gamma)-G(\gamma)
&=-\frac{\gamma\sigma_1^2}{N}+\frac1{2N}\log(\gamma-\Lambda_1)
-\frac{u_1^2}{4N(\gamma-\Lambda_1)}\nonumber\\
&\qquad-\frac1{2N}\log\det(\bI-\bR_\gamma^{1/2}\bDel\bR_\gamma^{1/2})
+\frac1{4N}\langle\bu_{-1},\bR_\gamma\bDel\bR_\gamma\bu_{-1}\rangle+O(N^{-2+Cc_0}).
\end{align}

This holds for $\gamma>\max_i\Lambda_i+\eps$. The same estimates hold after a fixed number of derivatives. Since $G$ is strongly convex in this neighborhood, $\gamma_*(\sigma_1)=\gamma_*+O(N^{-1+Cc_0})$.

 We next expand the log-determinant term in the difference. Defining
 
 \[
\bR=(\gamma_*\bI-\bLambda_{-1})^{-1},\qquad
\bB=\sigma_1\bR^{1/2}\tilde{\bG}\bR^{1/2}/N,\qquad
D_2=N^{-2}\Tr(\bR\tilde{\bG}\bR\tilde{\bG}),
\]
we have
\begin{align*}
\Tr(\bB)&=\frac{\sigma_1}{N}\sum_{i\ne1}(\gamma_*-\Lambda_i)^{-1}\tg_{1ii},\\
\Tr(\bB^2)&=D_2\sigma_1^2,\\
\sum_{k\ge3}\frac1k\Tr(\bB^k)&=O(N^{-1+Cc_0}).
\end{align*}
For the last estimate, the cubic trace has mean zero and second moment $O(N^3)$ before multiplication by $(\sigma_1/N)^3$; the terms of order at least four are bounded by the operator norm. Gaussian moment bounds give these estimates with probability $1-e^{-N^{c'}}$ for some $c'>0$.
Thus, with high probability,
\[
-\frac12\log\det(\bI-\bB)
=\frac{\sigma_1}{2N}\sum_{i\ne1}(\gamma_*-\Lambda_i)^{-1}\tg_{1ii}
+\frac{D_2}{4}\sigma_1^2+O(N^{-1+Cc_0}).
\]
Furthermore, we can compute
\[
\bbE_{\tilde{\bG}}D_2
=\frac{3\beta_3^2}{2N^2}\big((\Tr\bR)^2+\Tr\bR^2\big),\qquad
\operatorname{Var}_{\tilde{\bG}}(D_2)=O(N^{-2}).
\]
Recall the stationarity condition $G'(\gamma_*)=0$, which yields
\[
\frac1{2N}\sum_{i=1}^N\frac1{\gamma_*-\Lambda_i}
=1-\frac1{4N}\sum_{i=1}^N\frac{u_i^2}{(\gamma_*-\Lambda_i)^2}.
\]
For $\|\bu\|\le N^{c_0}$, this gives $\Tr\bR=2N+O(N^{2c_0})$, and therefore
\begin{align}\label{eq:ED2}
\bbE_{\tilde{\bG}}D_2=6\beta_3^2+O(N^{-1+Cc_0}),\qquad
D_2-\bbE_{\tilde{\bG}}D_2=O(N^{-1+Cc_0}).
\end{align}

After substituting the above estimates into Eq.~\eqref{eq:GdiffFirst},
the following holds with probability at least
$1-\exp(-N^c)$,

\begin{align*}
N\big(G_{\sigma_1}(\gamma_*(\sigma_1))-G(\gamma_*)\big)
&=-\gamma_*\sigma_1^2+\frac12\log(\gamma_*-\Lambda_1)-\frac{u_1^2}{4(\gamma_*-\Lambda_1)}\\
&\qquad+\frac{\sigma_1}{2N}\sum_{i\ne1}(\gamma_*-\Lambda_i)^{-1}\tg_{1ii}
+\frac{D_2}{4}\sigma_1^2+O(N^{-1+Cc_0}).
\end{align*}
Here the field term is bounded using $\bu_{-1}^T\bR\tilde{\bG}\bR\bu_{-1}=O(N^{c_0}(1+\|\bu\|)^2)$, and the saddle displacement changes the exponent and prefactor by $O(N^{-1+Cc_0})$.
Using Eq.~\eqref{eq:Esigma1} and
\[
U(\sigma_1)=\frac N2\xi_{\ge3}(1)-\frac32\beta_3^2\sigma_1^2+O(N^{-1+Cc_0}),
\]
we obtain
\begin{align}\label{eq:annealed-slice-density}
e^{u_1\sigma_1+\Lambda_1\sigma_1^2}E(\sigma_1)
&=K_N\exp\big(h_1\sigma_1-D_1\sigma_1^2+O(N^{-1+Cc_0})\big),\\
D_i&:=\gamma_*-\Lambda_i,\qquad
h_1:=u_1+\frac1{2N}\sum_{i\ne1}\frac{\tg_{1ii}}{D_i},\nonumber\\
K_N&:=e^{N\xi_{\ge3}(1)/2}Z_{\le2}\sqrt{\frac{D_1}{\pi}}
\exp\Big(-\frac{u_1^2}{4D_1}\Big).\nonumber
\end{align}
In particular, the quadratic terms $D_2\sigma_1^2/4$ and $-3\beta_3^2\sigma_1^2/2$ cancel to the required order, and $g_{111}\sigma_1^3/N=O(N^{-1+Cc_0})$.

We may restrict the integrals in \eqref{eq:ApplyAnneal} to $[-a_N,a_N]$. Indeed, averaging over the retained cubic variables gives $e^{N\xi_{\ge3}(1)/2}$ times the corresponding quadratic integrals, whose tails are bounded by $e^{-ca_N^2}Z_{\le2}$. On the other hand, integrating \eqref{eq:annealed-slice-density} over $|\sigma_1|\le1$ gives $\overline Z\ge e^{N\xi_{\ge3}(1)/2}Z_{\le2}e^{-CN^{2c_0}}$. Markov's inequality bounds the normalized tails, also with any fixed power of $|\sigma_1|$, by $e^{-N^{c'}}$, with probability $1-e^{-N^{c'}}$ after decreasing $c'$.

Thus, for $j=0,1$,
\begin{align*}
\int\sigma_1^je^{u_1\sigma_1+\Lambda_1\sigma_1^2}E(\sigma_1)\de\sigma_1
&=K_N\left\{\int_{\bbR}\sigma_1^je^{h_1\sigma_1-D_1\sigma_1^2}\de\sigma_1
+O(N^{-1+Cc_0})\int_{\bbR}e^{h_1\sigma_1-D_1\sigma_1^2}\de\sigma_1\right\}.
\end{align*}
Evaluating the Gaussian integrals, we get
\begin{align}\label{eq:cal-E}
\frac{\int\sigma_1e^{u_1\sigma_1+\Lambda_1\sigma_1^2}E(\sigma_1)\de\sigma_1}
{\int e^{u_1\sigma_1+\Lambda_1\sigma_1^2}E(\sigma_1)\de\sigma_1}
&=\frac{u_1+(2N)^{-1}\sum_{i\ne1}(\gamma_*-\Lambda_i)^{-1}\tg_{1ii}}
{2(\gamma_*-\Lambda_1)}+O(N^{-1+Cc_0}).
\end{align}
Since $\tg_{1ii}=3NW^{(3)}_{1ii}$, this gives the cubic coefficient $3/(4D_1)$. Including $i=1$ in the sum changes the answer by $O(N^{-1+Cc_0})$.
The claim now follows from Eq.~\eqref{eq:ApplyAnneal}, taking $c_0$ sufficiently small and then decreasing $c$.

\end{proof}
Finally, we prove Lemma \ref{lem:anneal}. The main idea is that
the error in annealing can be controlled by accurate estimates of certain quantities involving overlaps
under the quadratic model on $\bsig_{-1}$. These estimates follow from the Laplace
transform and an expansion of the dependence on $\sigma_1$.
\begin{proof}[Proof of Lemma \ref{lem:anneal}]
Define
\begin{align*}
W(\bsig^1,\bsig^2) &:= \frac{1}{N}\bbE\Big\{\big(H_3(\bsig_{-1}^1)+H_{\ge 4}(\bsig^1)\big)\big(H_3(\bsig_{-1}^2)+H_{\ge 4}(\bsig^2)\big)\Big\}\\
&=\beta_{3}^{2}\langle\bsig_{-1}^{1},\bsig_{-1}^{2}\rangle_N^{3}
+\xi_{\ge4}\big(\langle\bsig_{-1}^{1},\bsig_{-1}^{2}\rangle_N+\sigma_1^{1}\sigma_1^{2}/N\big)\, ,
\end{align*}
where, with an abuse of notation,
$H_3(\bsig_{-1}^a):=N^{-1}\sum_{i,j,k=2}^Ng_{ijk}\sigma^a_i\sigma^a_j\sigma^a_k$ (and a similar notation will be used for $H_{\le 2}(\bsig_{-1}^a)$
below). Note that $W(\bsig^1,\bsig^1) = \xi_{\ge4}(1) + \beta_3^2 (1 - (\sigma_1^1)^2/N)^3$.

For a deterministic Borel set $U\subseteq S_N^2$, define
\[
Q(U)=\int_U\sigma_1^1\sigma_1^2
\prod_{j=1}^2\big(\bbE_{\bg_{3-},\bg_{\ge4}}e^{H(\bsig^j)}\big)
\big(e^{NW(\bsig^1,\bsig^2)}-1\big)\mu_0^{\otimes2}(\de\bsig).
\]

Write $T=T(\delta)$ and
\[
X_T=\int_T\sigma_1e^{H(\bsig)}\mu_0(\de\bsig),\qquad
\overline X=\int\sigma_1\bbE_{\bg_{3-},\bg_{\ge4}}e^{H(\bsig)}\mu_0(\de\bsig),\qquad
\overline Z=\int\bbE_{\bg_{3-},\bg_{\ge4}}e^{H(\bsig)}\mu_0(\de\bsig).
\]
Choose $c>0$ sufficiently small, and then $c_0$ sufficiently small given $c$. We condition on the quadratic disorder, outside an event of probability $e^{-cN}$; the typicality errors in \cref{lem:truncate,lem:key-error} may first be conditioned by Markov's inequality. By \cref{lemma:A-T-comparison,lem:OrthogonalTuples}, we may then replace $T^2$ by $A_2=A_2(N^{-1/2+c})$ in $\bbE_{\bg_{3-},\bg_{\ge4}}X_T^2$. Markov's inequality over the retained cubic variables and the quadratic lower bound for $\overline Z$ give, for some $c'>0$, with probability $1-e^{-N^{c'}}$,
\begin{align*}
\bbE_{\bg_{3-},\bg_{\ge4}}X_T^2
&=\int_{A_2}\sigma_1^1\sigma_1^2
\prod_{j=1}^2\big(\bbE_{\bg_{3-},\bg_{\ge4}}e^{H(\bsig^j)}\big)e^{NW(\bsig^1,\bsig^2)}\mu_0^{\otimes2}(\de\bsig)
+O(e^{-N^{c'}}\overline Z^2),\\
|\bbE_{\bg_{3-},\bg_{\ge4}}X_T-\overline X|&\le e^{-N^{c'}}\overline Z.
\end{align*}
For $\overline X^2$, the same restriction follows from \cref{lem:overlap-2}, uniformly in the slice perturbations below, after discarding $|\sigma_1^j|>a_N$. Hence
\begin{align}\label{eq:HenceQ2}
\bbE_{\bg_{3-},\bg_{\ge4}}(X_T-\overline X)^2
\le Q(A_2)+e^{-N^{c'}}\overline Z^2.
\end{align}

By Taylor expansion, always using the shorthand $A_2=A_2(N^{-1/2+c})$,

\begin{align*}
Q(A_2)
&=\int_{A_2}\sigma_1^1\sigma_1^2\prod_{j=1}^2\big(\bbE_{\bg_{3-},\bg_{\ge4}}e^{H(\bsig^j)}\big)
\left\{\sum_{\ell=1}^L\frac{(NW(\bsig^1,\bsig^2))^\ell}{\ell!}\right\}
\mu_0^{\otimes2}(\de\bsig)+O(N^{-3}\overline Z^2).
\end{align*}
Here $L$ is sufficiently large. To make this expansion, we first restrict to $|\sigma_1^1|,|\sigma_1^2|\le a_N=N^{2c_0}\log N$. For the $e^{NW}$ term, averaging the absolute integrand over the retained cubic variables gives the full covariance $\xi_{\ge3}(\langle\bsig^1,\bsig^2\rangle_N)$, whose product with $N$ is $O(N^{-1/2+3c})$ on $A_2$. The quadratic tails, Markov's inequality, and the lower bound for $\overline Z$ give an error $O(e^{-N^{c'}}\overline Z^2)$. For the $-1$ term and the Taylor terms, the same bound follows from the tails in \eqref{eq:annealed-slice-density}.

We estimate each term
\begin{align}\label{eq:Taylor-T-ell}
T_\ell(a,b)&:=\int_{A_2}\sigma_1^1\sigma_1^2\prod_{j=1}^2\big(\bbE_{\bg_{3-},\bg_{\ge4}}e^{H(\bsig^j)}\big)
\cdot N^\ell\langle\bsig_{-1}^1,\bsig_{-1}^2\rangle_N^a
(\sigma_1^1\sigma_1^2/N)^b\,\mu_0^{\otimes2}(\de\bsig).
\end{align}
We can restrict ourselves to terms with $a\ge3\ell$ and $b=0$, or $a+b\ge3\ell+1$, since these are the terms that can arise in $Q(A_2)$.
Truncating the series for $\xi_{\ge4}$ at a sufficiently large fixed degree makes an additional error $O(N^{-3}\overline Z^2)$.
Let $\widehat T_\ell(a,b)$ be the same integral over $S_N^2$.
By Lemma \ref{lem:overlap-2}, applied to the quadratic slices below,
\[
|T_\ell(a,b)-\widehat T_\ell(a,b)|\le e^{-N^{c'}}\overline Z^2.
\]

Write $s=\sigma_1^1$, $t=\sigma_1^2$, and $f(s)=e^{u_1s+\Lambda_1s^2}E(s)$, so that $\overline Z=\int f(s)\,\de s$. Let $\nu_s$ be the normalized quadratic Gibbs measure on the sphere of radius $\sqrt{N-s^2}$, with matrix $\bLambda_{-1}+s\tilde{\bG}/N$ and field $\bu_{-1}+s^2\boldsymbol d/N$. Then
\begin{align*}
M_a(s,t)&:=\int\langle\bsig_{-1}^1,\bsig_{-1}^2\rangle^a\,
\nu_s(\de\bsig_{-1}^1)\nu_t(\de\bsig_{-1}^2),\\
\widehat T_\ell(a,b)
&=N^{\ell-a-b}\int_{[-a_N,a_N]^2}s^{b+1}t^{b+1}M_a(s,t)f(s)f(t)\,\de s\,\de t
+O(e^{-N^{c'}}\overline Z^2).
\end{align*}
Applying Lemma \ref{lem:laplace-moments-2}, we have

\begin{align}\label{eq:overlap-parity-bound}
|M_a(s,t)|&\le N^{\lfloor a/2\rfloor+C_ac_0},\nonumber\\
|M_a(s,t)-M_a(-s,t)|&\le N^{\lfloor a/2\rfloor+C_ac_0}\frac{|s|}{\sqrt N},\\
|M_a(s,t)-M_a(-s,t)-M_a(s,-t)+M_a(-s,-t)|
&\le N^{\lfloor a/2\rfloor+C_ac_0}\frac{|st|}{N},\nonumber
\end{align}
and the analogous estimate with $s,t$ interchanged. Indeed, the radius and field are even in each variable, whereas the matrix perturbations have derivatives of operator norm $O(N^{-1/2})$; we apply \eqref{eq:quadratic-moment-derivatives} and the mean value theorem.

Note that Eq.~\eqref{eq:annealed-slice-density} and its tail estimate give, for each fixed $j\ge0$,
\begin{align}\label{eq:annealed-odd-weight}
\int |s|^j f(s)\,\de s&\le N^{C_jc_0}\overline Z,\nonumber\\
|f(s)-f(-s)|&\le N^{Cc_0}(|u_1|+N^{-1/2})(1+|s|)\big(f(s)+f(-s)\big),\qquad |s|\le a_N.
\end{align}
When $b$ is even, $s^{b+1}$ and $t^{b+1}$ are both odd. Applying \eqref{eq:annealed-odd-weight} to the even parts of $M_a$, and \eqref{eq:overlap-parity-bound} to its odd parts, gives a factor $|u_1|+N^{-1/2}$ from each variable. When $b$ is odd, we bound the integral absolutely.
Therefore, under the assumption $\|\bu\|\le N^{c_0}$, for $\ell\le L$,

\begin{align*}
\frac{|\widehat T_\ell(a,b)|}{\overline Z^2}
&\le N^{\ell-a-b+\lfloor a/2\rfloor+Cc_0}
\begin{cases}
(|u_1|+N^{-1/2})^2,&b\text{ even},\\
1,&b\text{ odd},
\end{cases}\\
&\le N^{Cc_0}O_L(N^{-2}+N^{-3/2}|u_1|+N^{-1}|u_1|^2)
=N^{Cc_0}O_L(N^{-1}|u_1|^2+N^{-2}),
\end{align*}
where in the last step we used the fact that $\ell\ge1$, and $a\ge3\ell$ when $b=0$, or $a+b\ge3\ell+1$ otherwise.
Combining the terms in Eq.~\eqref{eq:Taylor-T-ell}, we obtain
\[
Q(A_2)\le N^{Cc_0}O(N^{-1}|u_1|^2+N^{-2})\overline Z^2,
\]
and therefore, using Eq.~\eqref{eq:HenceQ2}, with probability at least $1-N^{-c}$,
\[
|X_T-\overline X|\le N^{-1/2+c}(|u_1|+N^{-1/2})\overline Z.
\]
Here we choose $c_0$ sufficiently small given $c$ and use Chebyshev's inequality. The same calculation with the observable $1$ gives a relative second moment $O(N^{-1+Cc_0})$, proving \eqref{eq:partial-anneal-denominator}. Finally, Lemma \ref{lem:truncate} removes the restriction to $T$.

\end{proof}

We note that (\ref{eq:Magn-No-Field}) in Lemma \ref{lem:high-acc-anneal} immediately gives the following high
probability bound on the magnetization.
\begin{lem}
\label{lem:high-prob}For any $\eps,C>0$, there exists $c_0>0$ such that, for $\|\bu\| \le N^{c_0}$, with probability at
least $1-N^{-C}$, we have
\begin{align}
\left\|\int \bsig\mu(\de\bsig)\right\|_{}^{2}\le N^{\eps}\, , \label{eq:highprob-mag}
\end{align}
for $N$ sufficiently large.
\end{lem}
\begin{proof}
We work as before in the eigenbasis of $\bW^{(2)}$. By \cref{eq:Magn-No-Field,eq:HighAcc-PartitionFun,eq:Z3T} and Lemma \ref{lem:truncate}, choosing the decay parameter $L$ sufficiently large, the integer
moment order $k>L$ sufficiently large for the coordinate union bound,
and then $c_0$ sufficiently small, with probability at least $1-N^{-C}$,
\[
|\langle\sigma_i\rangle|\le N^{\eps/4}(|u_i|+N^{-1/2})\qquad(i\in[N]).
\]
Squaring and summing over $i$, we obtain, after decreasing $c_0$,
\[
\|\langle\bsig\rangle\|^2\le 2N^{\eps/2}(1+\|\bu\|^2)\le N^\eps.
\]
\end{proof}

Theorem \ref{lem:correction-band} now follows.
\begin{proof}[Proof of Theorem \ref{lem:correction-band}]
Let $\be=\langle\bsig\rangle-\hbm$. We condition on $\bW^{(2)}$, outside an event of probability $e^{-cN}$; the proofs of \cref{lem:spin,lem:high-prob} give the same conditional probability bounds. By \cref{lem:spin}, for each $i$, outside an event of conditional probability $N^{-c}$,
\[
|e_i|^2\le N^{-c}(|u_i|^2+N^{-1}),
\]
after decreasing $c$. On the other hand, the coordinate estimate in the proof of \cref{lem:high-prob}, together with Gaussian concentration of the cubic correction, gives an event $\cE_D$ of conditional probability at least $1-N^{-D}$ on which
\[
|e_i|^2\le N^\eps(|u_i|^2+N^{-1})\quad(i\in[N]),\qquad
\|\be\|^2\le N^\eps(1+\|\bu\|^2).
\]
Consequently, summing the coordinate expectations,
\[
\bbE_{\ge3}[\|\be\|^2\bfone_{\cE_D}]
\le C N^{-c+\eps}(1+\|\bu\|^2).
\]
For $\alpha\ge2$, multiply by the bound on $\|\be\|^{\alpha-2}$ on $\cE_D$. The complementary event contributes $O(N^{-D/2+C_\alpha})$: indeed, $\|\langle\bsig\rangle\|\le\sqrt N$, the stationarity equation gives $(\gamma_*-\Lambda_i)^{-1}\le2N$, and the cubic coefficients have Gaussian moments. Choosing $D$ large, and then $\eps,c_0$ small, gives $\bbE\|\be\|^\alpha=O(N^{-c'})$ for some $c'>0$. The theorem follows.
\end{proof}

\subsection*{Acknowledgements}

BH was supported by a Google PhD Fellowship, NSF CAREER grant DMS-1940092, and the 
Solomon Buchsbaum Research Fund at MIT.
AM was partially supported by the NSF Awards DMS MFAI-2501597
and DMS MFAI-2601871.
HTP was supported by a Clay Research Fellowship, a Sloan Research Fellowship, 
NSF Career grant DMS-2543870, and a Bell System Fellowship.

\bibliographystyle{alpha}
\bibliography{bib}

\end{document}